\documentclass[12pt,a4paper,leqno]{amsart}

\usepackage[latin1]{inputenc}
\usepackage[T1]{fontenc}
\usepackage{amsfonts}
\usepackage{amsmath}
\usepackage{amssymb}
\usepackage{eurosym}
\usepackage{mathrsfs}
\usepackage{palatino}
\usepackage{color}
\usepackage{esint}
\usepackage{url}

 \usepackage{amsrefs}
 
\usepackage{hyperref} 
\hypersetup{
    colorlinks=false,       
    linkcolor=blue,          
    citecolor=magenta,        
    filecolor=magenta,      
    urlcolor=cyan           
}

\newcommand{\R}{\mathbb{R}}

\newcommand{\N}{\mathbb{N}}

\newcommand{\Z}{\mathbb{Z}}

\numberwithin{equation}{section}

\newcommand{\Car}[0]{\operatorname{Car}}

\newcommand{\prob}{\mathbb{P}}

\newcommand{\D}[0]{\mathcal{D}}
\newcommand{\eps}[0]{\varepsilon}

\newcommand{\car}[0]{\operatorname{Car}}
\newcommand{\mcar}[0]{\widetilde{\operatorname{Car}}}

\swapnumbers
\theoremstyle{plain}
\newtheorem{thm}[equation]{Theorem}
\newtheorem{lem}[equation]{Lemma}
\newtheorem{prop}[equation]{Proposition}

\theoremstyle{definition}

\theoremstyle{remark}
\newtheorem{rem}[equation]{Remark}

\author{Henri Martikainen}
\address[H.M.]{Department of Mathematics and Statistics, University of Helsinki, P.O.B. 68, FI-00014 Helsinki, Finland}
\email{henri.martikainen@helsinki.fi}
\thanks{Research of H.M. is supported by the Academy of Finland through the grant
Multiparameter dyadic harmonic analysis and probabilistic methods. The hospitality of M. Lacey and the School of Mathematics of Georgia Tech, where part of this research was carried,
is gratefully acknowledged by H.M}

\author{Mihalis Mourgoglou}
\address[M.M.]{Institut des Hautes Etudes Scientifiques, 91440 Bures-sur-Yvette, France}
\email{mourgoglou@ihes.fr}
\thanks{Research of M.M. is supported by the Institut des Hautes \'{E}tudes Scientifiques and was partially supported by the ANR project "Harmonic analysis at its boundaries" ANR-12-BS01-0013-01. The hospitality of T. Hyt\"{o}nen and the Department of Mathematics and Statistics, University of Helsinki, where part of this research was carried,
is gratefully acknowledged by M.M}

\makeatletter
\@namedef{subjclassname@2010}{%
  \textup{2010} Mathematics Subject Classification}
\makeatother


\title[Boundedness of non-homogeneous square functions]{Boundedness of non-homogeneous square functions and $L^q$ type testing conditions with $q \in (1,2)$}

\begin{document}
\maketitle
\begin{abstract}
We continue the study of local $Tb$ theorems for square functions defined in the upper half-space $(\R^{n+1}_+, \mu \times dt/t)$. Here $\mu$ is allowed to be a non-homogeneous measure in $\R^n$.
In this paper we prove a boundedness result assuming local $L^q$ type testing conditions in the difficult range $q \in (1,2)$. Our theorem is a non-homogeneous version of a result of S. Hofmann
valid for the Lebesgue measure.
It is also an extension of the recent results of M. Lacey and the first named author where non-homogeneous local $L^2$ testing conditions have been considered.
\end{abstract}

\section{Introduction}
We study the boundedness of the vertical square function
\begin{displaymath}
Vf(x) = \left( \int_{0}^{\infty} |\theta_{t}f(x)|^{2} \, \frac{dt}{t} \right)^{1/2}.
\end{displaymath} 
Here the linear operators $\theta_{t}$, $t > 0$, have the form
\begin{equation} \theta_{t}f(x) = \int_{\R^{n}} s_{t}(x,y)f(y) \, d\mu(y). 
\label{Intro:theta}
\end{equation}
The appearing measure $\mu$ is a Borel measure in $\R^n$ which is only assumed to satisfy, for some $m$, the upper bound 
\begin{equation*}\label{powerBound} \mu(B(x,r)) \lesssim r^{m}, \qquad x \in \R^{n}, \; r > 0.
\end{equation*}
Moreover, for some $\alpha > 0$, the kernels $s_{t}$ satisfy the size and continuity conditions
\begin{equation}\label{eq:size}
|s_t(x,y)| \lesssim \frac{t^{\alpha}}{(t+|x-y|)^{m+\alpha}}
\end{equation}
and
\begin{equation}\label{eq:yhol}
|s_t(x,y) - s_t(x,z)| \lesssim \frac{|y-z|^{\alpha}}{(t+|x-y|)^{m+\alpha}} \end{equation}
whenever $|y-z| < t/2$.

The following is our main theorem.
\begin{thm}\label{thm:main}
Let $q \in (1,2)$ be a fixed number. Assume that to every cube $Q \subset \R^n$ there is associated a test function $b_Q$
satisfying that
\begin{itemize}
\item[(1)] spt$\,b_Q \subset Q$;
\item[(2)] $\int_Q b_Q \,d\mu = \mu(Q)$;
\item[(3)] $\|b_Q\|_{L^q(\mu)}^q \lesssim\mu(Q)$;
\item[(4)] \begin{displaymath}
\int_Q \Big( \int_0^{\ell(Q)} |\theta_t b_Q(x)|^2 \frac{dt}{t} \Big)^{q/2} \,d\mu(x) \lesssim \mu(Q).
\end{displaymath}
\end{itemize}
Then we have that
\begin{displaymath}
\|V\|_{L^q(\mu) \to L^q(\mu)} \lesssim 1.
\end{displaymath}
\end{thm}

\begin{rem}
Suppose we also have the $x$-continuity
\begin{equation}\label{eq:xhol}
|s_t(x,y) - s_t(z,y)| \lesssim \frac{|x-z|^{\alpha}}{(t+|x-y|)^{m+\alpha}} \end{equation}
whenever $|x-z| < t/2$. 
Then we have that
\begin{displaymath}
\|V\|_{L^2(\mu) \to L^2(\mu)} \lesssim 1.
\end{displaymath}
It should be noted that an example from \cite{MMO} shows that when dealing with the vertical square function (as we are here)
one cannot derive the $L^2(\mu)$ estimate from the $L^q(\mu)$ estimate without $x$-continuity. This fails even in the case that $\mu$ is the Lebesgue measure.
\end{rem}

Hofmann \cite{Ho} proved the $L^2$ boundedness of the square function
under these local $L^q$ testing conditions in the case that $\mu$ is the Lebesgue measure. In the non-homogeneous case Lacey and the first named author \cite{LM2}
proved the $L^2$ boundedness but only with local $L^2$ testing conditions. Our main theorem is an extension of these two state of the art results. Indeed, we consider general measures and general exponents
simultaneously. The aforementioned two references are the most obvious predecessors of our main theorem, but the whole story up to this point is rather long.

One can consider $Tb$ theorems at least for square functions and Calder\'on--Zygmund operators.
Then they can be global or local. And if they are local, they can be with the easier $L^{\infty}/\textup{BMO}/T^{2,\infty}$ type testing assumptions, or with the more general $L^s$, $s < \infty$, type assumptions.
Moreover, in the latter case the range of the exponents (in the Calder\'on--Zygmund world more than one set of testing functions appear) one can use is a very significant problem.
Lastly, the fact that whether one considers the homogeneous or
non-homogeneous theory is a major factor. All of these big story arcs are relevant for the context of the current paper. We now try to give at least some of the key references of local $Tb$ theorems.

The first local $Tb$ theorem, with $L^{\infty}$ control of the test functions and their images, is by Christ \cite{Ch}.
Nazarov, Treil and Volberg \cite{NTVa} proved a non-homogeneous version of this theorem. The point compared to the global $Tb$ theorems is as follows. The 
accretivity of a given test function $b_Q$ is only assumed on its supporting cube $Q$, i.e., $|\int_Q b_Q\,d\mu| \gtrsim \mu(Q)$. While in a global $Tb$ one needs
a function which is simultaneously accretive on all scales. But the remaining conditions are still completely scale invariant: $b_Q \in L^{\infty}(\mu)$ and $Tb_Q \in L^{\infty}(\mu)$.
This scale invariance of the testing conditions is the main thing one wants to get rid of.

The non-scale-invariant $L^s$ type testing conditions were introduced by Auscher, Hofmann, Muscalu, Tao and Thiele \cite{AHMTT}.
Their theorem is for perfect dyadic singular integral operators and the assumptions 
are of the form $\int_Q |b^1_Q|^p \lesssim |Q|$, $\int_Q |b^2_Q|^q \lesssim |Q|$, $\int_Q |Tb^1_Q|^{q'} \lesssim |Q|$ and $\int_Q |T^*b^2_Q|^{p'} \lesssim |Q|$, $1 < p, q \le \infty$.
Extending the result to general Calder\'on--Zygmund operators is complicated (it is almost done by now -- but not completely).
Hofmann \cite{Ho1} established the result for general operators but only assuming the existence of $L^{2+\epsilon}$ test functions mapping to $L^2$.
Auscher and Yang \cite{AY} removed the $\epsilon$ by proving the theorem in the sub-dual case $1/p + 1/q \le 1$. Auscher and Routin \cite{AR} considered the general case
under some additional assumptions. The full super-dual case $1/p + 1/q  > 1$ is by Hyt\"onen and Nazarov \cite{HN}, but  even then with the additional buffer assumption
$\int_{2Q} |Tb^1_Q|^{q'} \lesssim |Q|$ and $\int_{2Q} |T^*b^2_Q|^{p'} \lesssim |Q|$. 

This was the main story for the Calder\'on--Zygmund operators for doubling measures. For square functions the situation is a bit more clear with the need for only one exponent $q$.
The case $q = 2$ is implicit in the Kato square root papers \cite{HMc}, \cite{HLM}, \cite{AHLMT} and explicitly stated and proved in \cite{A} and \cite{Ho2}. The case $q > 2$ is weaker than this. The hardest case $q \in (1,2)$ is due to Hofmann \cite{Ho} as already mentioned.
Some key applications really need the fact that one can push the integrability of the test functions to $1+\epsilon$ (see again \cite{Ho}).

The non-homogeneous world is yet another story. The whole usage of these non-scale-invariant testing conditions is a huge source of problem in this context.
One reason lies in the fact that even if we have performed a stopping time argument which gives us that a fixed test function $b_F$ behaves nicely on a cube $Q$, for example that $\int_Q |b_F|^2\, d\mu \lesssim \mu(Q)$, we cannot say much what happens in the stopping children of $Q$. That is, in a stopping child $Q'$ of $Q$ we cannot use the simple argument
\begin{displaymath}
\int_{Q'} |b_F|^2 \,d\mu \le \int_{Q} |b_F|^2 \,d\mu \lesssim \mu(Q) \lesssim \mu(Q')
\end{displaymath}
which would only be available if $\mu$ would be doubling. The non-homogeneous case $q = 2$ for square functions is the very recent work of Lacey and the first named author \cite{LM2}.
The case $p=q=2$ for Calder\'on--Zygmund operators
is by the same authors \cite{LM1}. For relevant dyadic techniques see also the Lacey--V\"ah\"akangas papers \cite{LV} and \cite{LV1}, and Hyt\"onen--Martikainen \cite{HyM}. To recap the context, in this paper we consider non-homogeneous square functions and push $q$ to the range $q \in (1,2)$.

We still mention that the study of the boundedness of non-homogeneous square functions was initiated by the recent authors in \cite{MM}. This was a global $Tb$.
The key technique was the usage of good (in a probabilistic sense) Whitney regions.
A scale invariant local $Tb$ is by the current authors together with T. Orponen \cite{MMO}. In that paper we also study the end point theory, $L^p$ theory, and various
counter-examples (e.g. the failure of the change of aperture with general measures and the difference between conical and vertical square functions).

We conclude the introduction by a remark and setting up some notation.
\begin{rem}
If we define 
$$V_{\textup{loc}, q}= \Big[\mathop{\sup_{Q \subset \R^n}}_{Q\, \textup{cube}} \frac{1}{\mu(Q)}\int_Q \Big( \int_0^{\ell(Q)} |\theta_t b_Q(x)|^2 \frac{dt}{t} \Big)^{q/2}\,d\mu(x) \Big]^{1/q},$$
then $\|V\|_{L^q(\mu) \to L^q(\mu)} \lesssim 1 + V_{\textup{loc}, q}$, where the implicit constants depend on $n, m, \alpha$, the kernel constants and the constant in testing condition (3). In the proof we will not keep track of the dependence on anything else but $V_{\textup{loc}, q}$.

The local $Tb$ with $L^q$ testing conditions can be proved assuming only that $\mu(B(x,r)) \le \lambda(x,r)$ for some $\lambda\colon \R^n \times (0,\infty) \to (0,\infty)$ satisfying that
$r \mapsto \lambda(x,r)$ is non-decreasing and $\lambda(x, 2r) \le C_{\lambda}\lambda(x,r)$ for all $x \in \R^n$ and $r > 0$. In this case one only needs to replace the kernel estimates by
\begin{displaymath}
|s_t(x,y)| \lesssim \frac{t^{\alpha}}{t^{\alpha}\lambda(x,t) + |x-y|^{\alpha}\lambda(x, |x-y|)}
\end{displaymath}
and
\begin{displaymath}
|s_t(x,y) - s_t(x,z)| \lesssim \frac{|y-z|^{\alpha}}{t^{\alpha}\lambda(x,t) + |x-y|^{\alpha}\lambda(x, |x-y|)}
\end{displaymath}
whenever $|y-z| < t/2$. This is done in the global situation in \cite{MM}. Here we skip the required modifications.
Such formalism
lets one capture the doubling theory as a by-product, and allows some more general upper bounds than $r^m$.
\end{rem}

\subsection{Notation}
We write $A \lesssim B$, if there is a constant $C>0$ so that $A \leq C B$. We may also write $A \approx B$ if $B \lesssim A \lesssim B$. For a number $a$ we write
$a \sim 2^k$ if $2^k \leq a <2^{k+1}$.

We then set some dyadic notation. Consider a dyadic grid $\D$ in $\R^n$. For $Q, R \in \D$ we use the following notation:
\begin{itemize} 
\item $\ell(Q)$ is the side-length of $Q$;
\item $d(Q,R)$ denotes the distance between the cubes $Q$ and $R$;
\item $D(Q,R):=d(Q,R)+\ell(Q)+\ell(R)$ is the long distance;
\item $\widehat{Q}=Q\times (0, \ell(Q))$ is the Carleson box associated with $Q$;
\item $W_Q=Q\times (\ell(Q)/2, \ell(Q))$ is the Whitney region associated with $Q$;
\item $\text{ch}(Q)=\{Q' \in \D: Q' \subset Q, \ell(Q') = \ell(Q)/2\}$;
\item $\text{gen}(Q)$ is determined by $\ell(Q)=2^{\text{gen}(Q)}$;
\item $Q^{(k)} \in\D$ is the unique cube for which $\ell(Q^{(k)})=2^k \ell(Q)$ and $Q \subset Q^{(k)}$;
\item $\langle f \rangle_Q = \mu(Q)^{-1}\int_Q f\,d\mu$.
\end{itemize}

\section{Structure of the proof and basic reductions}
\subsection{Reduction to a priori bounded operators $V$}\label{reduction.apriori.bounded}
In this subsection we say the following. Suppose we have proved the $L^q(\mu)$ bound of Theorem \ref{thm:main}, i.e., the quantitative bound $\|V\|_{L^q(\mu) \to L^q(\mu)} \lesssim 1+V_{\textup{loc}, q}$,
under the additional a priori finiteness assumption $\|V\|_{L^q(\mu) \to L^q(\mu)} < \infty$. Then the $L^q(\mu)$ bound of Theorem \ref{thm:main} automatically follows without the a priori assumption.

To this end,
define $s^i_t(x,y) = s_t(x,y)$ if $1/i \le t \le i$, and $s_t^i(x,y) = 0$ otherwise. These kernels are clearly in our original class -- they satisfy \eqref{eq:size} and \eqref{eq:yhol} with kernel constants bounded
by those of $V$. Define
\begin{displaymath}
V_if(x) := \left( \int_{1/i}^{i} |\theta_{t}f(x)|^{2} \, \frac{dt}{t} \right)^{1/2} = \left( \int_{0}^{\infty} |\theta_{t}^if(x)|^{2} \, \frac{dt}{t} \right)^{1/2},
\end{displaymath}
where
\begin{displaymath} \theta_{t}^if(x) = \int_{\R^{n}} s_{t}^i(x,y)f(y) \, d\mu(y). 
\end{displaymath}
Let us note that the $V_i$ are bounded operators on $L^q(\mu)$. Let
\begin{displaymath}
M_{\mu}f(x) = \sup_{r>0} \frac{1}{\mu(B(x,r))} \int_{B(x,r)} |f|\,d\mu(y).
\end{displaymath}
This centred maximal function is a bounded operator on $L^p(\mu)$ for
every $p \in (1,\infty)$. Notice that $|\theta_t f(x)| \lesssim M_{\mu} f(x)$ for every $t > 0$ and $x \in \R^n$. Using this we see that
$\|V_i\|_{L^q(\mu) \to L^q(\mu)} \le [2 \log i]^{1/2} \|M_{\mu}\|_{L^q(\mu) \to L^q(\mu)} < \infty$. 

By monotone convergence we have that
\begin{align*}
\|Vf\|_{L^q(\mu)} &= \lim_{i \to \infty} \|V_if\|_{L^q(\mu)} \\ &
\le \limsup_{i \to \infty} \|V_i\|_{L^q(\mu) \to L^q(\mu)} \|f\|_{L^q(\mu)} \\
&\lesssim \limsup_{i \to \infty} (1+V_{\textup{loc}, q}^i) \|f\|_{L^q(\mu)} \\
&\le (1 + V_{\textup{loc}, q}) \|f\|_{L^q(\mu)}.
\end{align*}
So it suffices to prove Theorem \ref{thm:main} under the assumption $\|V\|_{L^q(\mu) \to L^q(\mu)} < \infty$ -- a piece of information that will be used purely in a qualitative way.

\subsection{Reduction to a q-Carleson estimate}
We begin by stating a $T1$ in $L^q(\mu)$ (the case $q = 2$ is in \cite{MM}). The proof of this $T1$ is indicated in Appendix \ref{App:T1}.
Define, say for $\lambda \geq 3$,
\begin{displaymath}
\car_V(q, \lambda) := \mathop{\sup_{Q \subset \R^n}}_{\textup{cube}} \Big[ \frac{1}{\mu(\lambda Q)} \int_{Q} \Big( \int_0^{\ell(Q)} |\theta_t 1_{Q}(x)|^2 \frac{dt}{t} \Big)^{q/2} \,d\mu(x)\Big]^{1/q}
\end{displaymath}
and
\begin{displaymath}
\mcar_V(q, \lambda) :=  \mathop{\sup_{Q \subset \R^n}}_{\textup{cube}} \Big[ \frac{1}{\mu(\lambda Q)} \int_{Q} \Big( \int_0^{\ell(Q)} |\theta_t 1(x)|^2 \frac{dt}{t} \Big)^{q/2} \,d\mu(x)\Big]^{1/q}.
\end{displaymath}
Then, for $q \in (1,2]$, we have that there holds that 
\begin{equation}\label{eq:Intro.qt1}
\|V\|_{L^q(\mu) \to L^q(\mu)} \le C_1(1+ \mcar_V(q,9)) \le C_2(1 + \car_V(q,3)).
\end{equation}
Assuming the existence of the $L^q$ test functions as in Theorem \ref{thm:main} we then prove that
\begin{equation}\label{Intro:Carleson}
\car_V(q,3) \le C_3 (1 +V_{\textup{loc}, q}) + C_2^{-1} \|V\|_{L^q(\mu) \to L^q(\mu)}/2.
\end{equation}
We call this the key inequality.
Combining \eqref{eq:Intro.qt1} and \eqref{Intro:Carleson} gives that
\begin{displaymath}
\|V\|_{L^q(\mu) \to L^q(\mu)} \le C (1 +V_{\textup{loc}, q})  +  \|V\|_{L^q(\mu) \to L^q(\mu)}/2
\end{displaymath}
ending the proof.

We will now start the proof of the key inequality \eqref{Intro:Carleson}.
This task is completed in Section \ref{section:nested}.
In Appendix \ref{App:T1} we indicate the proof of the $T1$ theorem in $L^q(\mu)$, i.e., the first estimate of \eqref{eq:Intro.qt1}.

\section{Random and stopping cubes/ Martingale difference operators}

\subsection{Random dyadic grids}
At this point we need to set up the basic notation for
random dyadic grids (these facts are essentially presented in this way by Hyt\"onen \cite{Hy}).

Let $\D_0$ denote the standard dyadic grid, consisting of all the cubes of the form $2^k(\ell + [0,1)^n)$, where $k \in \Z$ and $\ell \in \Z^n$. We also denote $\D_{0, k}=\{ Q \in \D_0: \ell(Q)=2^k \}$. A generic dyadic grid, parametrized by $w \in (\{0,1\}^n)^\Z$, is of the form $\D(w)= \cup_{k \in \Z} \D_k(w)$, where $\D_k(w)=\{Q_0+x_k(w): Q_0 \in \D_{0, k}\}$ and $x_k(w)=\sum_{j<k} w_j 2^j$.
The notation $Q_0 + w := Q_0 + \sum_{j<k} w_j 2^j$, $Q_0 \in \mathcal{D}_{0,k}$, is convenient.
We get random dyadic grids by placing the natural product probability measure $P_w$ on $(\{0,1\}^n)^\Z$ (thus the coordinate functions $w_j$ are independent and $P_w(w_j=\eta)=2^{-n}$ if $\eta \in \{0,1\}^n$). 

We fix the constant $\gamma \in (0,1)$ to be so small that
$$\gamma \leq \alpha/(2m+2\alpha) \quad \text{and} \quad m \gamma/(1- \gamma) \leq \alpha/4,$$ 
where $\alpha > 0$ appears in the kernel estimates and $m$ appears in $\mu(B(x,r)) \lesssim r^m$. A cube $R \in \mathcal{D}$ is called $\D$-bad if there exists another cube $Q \in \mathcal{D}$ so that $\ell(Q) \ge 2^r \ell(R)$ and $d(R, \partial Q) \le \ell(R)^{\gamma}\ell(Q)^{1-\gamma}$. Otherwise it is good. We denote the collections of good and bad cubes by $\D_{\textrm{good}}$ and $\D_{\textrm{bad}}$ respectively. The following properties are known (see e.g. \cite{Hy}).
\begin{itemize}
\item For a fixed $Q_0 \in \mathcal{D}_0$ the set $Q_0 + w$ depends on $w_j$ with $2^{j} < \ell(Q_0)$, while the goodness (or badness) of $Q_0 + w$ depends on $w_j$ with $2^{j} \ge \ell(Q_0)$. In particular, these notions are independent (meaning that
for any fixed $Q_0 \in \mathcal{D}_0$ the random variable $w \mapsto 1_{\textup{good}}(Q_0+w)$ and any random variable that depends only on the cube $Q_0+w$ as a set, like $w \mapsto \int_{Q_0+w} f\,d\mu$, are independent).
\item The probability $\pi_{\textrm{good}} := P_{w}(Q _0+ w \textrm{ is good})$ is independent of $Q_0 \in \mathcal{D}_0$. 
\item $\pi_{\textrm{bad}} :=1-\pi_{\textrm{good}} \lesssim 2^{-r \gamma}$, with the implicit constant independent of $r$.
\end{itemize}
The parameter $r \lesssim 1$ is a fixed constant which is at least so large that $2^{r(1-\gamma)} \ge 10$. 

The following lemma is stated without proof since the first part was proved on page 25 of \cite{Hy1} and the second is lemma 2.10 of \cite{LV}.
\begin{lem}\label{lemma.R<S(Q)} Let $Q \in \D$ and $R \in \D_{\textup{good}}$, and set $\theta(u):= \Big\lceil\frac{\gamma u +r}{1-\gamma}\Big\rceil$, $u \in \N$.
\begin{itemize}
\item[(1)] Assume $\ell(Q)<\ell(R)$. Let $\ell(R)/\ell(Q)=2^\ell$ and $D(Q,R)/\ell(R) \sim 2^j $ for $\ell\geq 1$ and $j \geq 0$. Then, there holds that 
$$R \subset Q^{(\ell+j+\theta(j))}.$$
\item[(2)] Assume $\ell(R)\leq \ell(Q)$. Let $\ell(Q)/\ell(R)=2^\ell$ and $D(Q,R)/\ell(Q) \sim 2^j $ for $\ell, j \geq 0$. Then there holds that
$$R \subset Q^{(j+\theta(j+\ell))}.$$
\end{itemize}
\end{lem}

\subsection{Collections of stopping cubes}
Let $\mathcal{D}$ be a dyadic grid in $\R^n$ and let $Q^* \in \mathcal{D}$ be a fixed dyadic cube with $\ell(Q^*) = 2^s$.
Set $\mathcal{F}_{Q^*}^0 = \{Q^*\}$ and let $\mathcal{F}_{Q^*}^1$ consist of the maximal cubes $Q\in \D$, $Q \subset Q^*$, for which at least one of the following two conditions holds:
\begin{enumerate}
\item $| \langle b_{Q^*}\rangle_Q | < 1/2$;
\item $\langle |b_{Q^*}|^q\rangle_Q >  2^{q'+1}A^{q'}$.
\end{enumerate}
Here $A$ is a constant such that $\|b_R\|_{L^q(\mu)}^q \le A\mu(R)$ for every cube $R \subset \R^n$.

Next, we repeat the previous procedure by replacing $Q^*$ with a fixed $Q \in \mathcal{F}_{Q^*}^1$. The combined collection of stopping cubes resulting from this is called $\mathcal{F}_{Q^*}^2$. This is continued and we set $\mathcal{F}_{Q^*}=\bigcup_{j=0}^\infty \mathcal{F}_{Q^*}^j$
Finally, for every $Q \in \D$, $Q \subset Q^*$, we let $Q^\alpha \in \mathcal{F}_{Q^*}$ be the minimal cube $R \in \mathcal{F}_{Q^*}$ for which $Q \subset R$. 

\begin{lem} If $F \in \mathcal{F}^{j}_{Q^*}$ for some $j\geq 0$, then there holds that
\begin{equation}
\mathop{\sum_{S \in \mathcal{F}^{j+1}_{Q^*}}}_{S \subset F} \mu(S) \leq \tau \mu(F), \,\, \tau := 1 - \frac{1}{2}\frac{A}{(2A)^{q'}} \in (0,1).
\label{StoppingFk}
\end{equation}
\end{lem}

\begin{proof}
Let $F \in \mathcal{F}_{Q^*}$. Consider a disjoint collection $\{Q^1_i\}_i \subset \D$ for which $Q^1_i \subset F$ and $| \langle b_{F}\rangle_{Q^1_i} | < 1/2$. We have that 
\begin{align*}
\mu(F)= \int_F b_F \,d\mu &=  \int_{F\setminus \bigcup_i Q^1_i} b_F \,d\mu + \sum_i \int_{Q^1_i} b_F \,d\mu\\
&\leq \mu \Big(F \setminus \bigcup_i Q^1_i \Big)^{1/q'}  \Big( \int_F|b_F|^q \,d\mu \Big)^{1/q} + \frac{1}{2} \sum_i \mu(Q^1_i)\\
& \leq A^{1/q}  \mu \Big(F \setminus \bigcup_i Q^1_i \Big)^{1/q'}  \mu(F)^{1/q} +\frac{1}{2}  \mu(F),
\end{align*}
which implies that
$$\mu(F) \leq \big(2 A^{1/q} \big)^{q'} \cdot \mu \Big(F \setminus \bigcup_i Q^1_i \Big) = \frac{\big(2 A \big)^{q'}}{A}   \Big[\mu(F) - \Big( \bigcup_i Q^1_i \Big)\Big].$$
Therefore, we obtain
$$\mu \Big(\bigcup_i Q^1_i \Big) \leq \Big(1-\frac{A}{(2A)^{q'}}\Big) \mu(F).$$

Next, we consider a disjoint collection $\{Q^2_i\}_i \subset \D$ for which $Q^2_i \subset F$ and $\langle |b_{F}|^q\rangle_{Q^2_i} >  2^{q'+1}A^{q'}$. Then, one may notice that
$$\mu \Big(\bigcup_i Q^2_i \Big)   \leq  2^{-q'-1}A^{-q'}  \int_F |b_F|^q \,d\mu \leq   \frac{1}{2}\frac{A}{(2A)^{q'}} \mu(F).$$ 
Combining the analysis we conclude that \eqref{StoppingFk} holds.
\end{proof}
The next lemma follows.
\begin{lem}\label{lem:mescar}
The following is a Carleson sequence: $\alpha_Q = 0$ if $Q$ is not from $\bigcup_j \mathcal{F}_{Q^*}^j$, and it equals $\mu(Q)$ otherwise.
This means that $\sum_{Q \subset R} a_Q \lesssim \mu(R)$ for every dyadic $R$.
\end{lem}
We now state the classical Carleson embedding theorem.
\begin{prop}
Given a Carleson sequence $(A_Q)_{Q\in \D}$ we have for every $f \in L^p(\mu)$, $1 < p < \infty$, that
$$\sum_{Q\in \D}  |\langle f \rangle_Q|^p A_Q \le C \|f\|^p_{L^p(\mu)}.$$
\end{prop}
\begin{rem} Note that $q$ is always reserved to be the fixed index $q \in (1,2)$ appearing in the testing conditions.
\end{rem}
The next proposition is a Carleson embedding on $L^p(\mu)$, where the Carleson condition itself depends on $p$. This kind of Carleson is also well-known, of course, but we state and prove this general version here for
the convenience of the reader.
\begin{prop}\label{CarlesonEmbeddingThm}
Let $\mathcal{D}$ be a dyadic grid in $\R^n$ and $p \in (1,2]$ be a fixed number.
Suppose that for every $Q \in \mathcal{D}$ we have a function $A_Q$ satisfying that spt$\,A_Q \subset Q$ and
\begin{equation}
\car_p( (A_Q)_{Q \in \mathcal{D}}) := \Big( \sup_{R \in \mathcal{D}} \frac{1}{\mu(R)} \int_R \Big[ \mathop{\sum_{Q \in \mathcal{D}}}_{Q \subset R} |A_Q(x)|^2\Big]^{p/2}\,d\mu(x) \Big)^{1/p} < \infty.
\label{eq:qCarlesonCondition}\end{equation}
Then we have that
\begin{equation}
\int \Big[ \sum_{Q \in \mathcal{D}} |\langle f \rangle_Q|^2 |A_Q(x)|^2 \Big]^{p/2}\,d\mu(x) \lesssim \car_p( (A_Q)_{Q \in \mathcal{D}})^p \|f\|_{L^p(\mu)}^p.
\label{eq:Carlesonembedding}
\end{equation}
\end{prop}
\begin{proof}
For each fixed $j \in \Z$ let $(R^i_j)_i$ denote the maximal $R \in \mathcal{D}$ for which $|\langle f \rangle_R| > 2^j$.
We have that
\begin{align*}
\int &\Big[ \sum_{Q \in \mathcal{D}} |\langle f \rangle_Q|^2 |A_Q(x)|^2 \Big]^{p/2}\,d\mu(x) \\
&= \int \Big[ \sum_{j\in \Z} \mathop{\sum_{Q \in \D}}_{|\langle f \rangle_Q| \sim 2^j} |\langle f \rangle_Q|^2 |A_Q(x)|^2 \Big]^{p/2}\,d\mu(x)\\
 &\lesssim \int \Big[ \sum_{j \in \Z} 2^{2j} \sum_i \mathop{\sum_{Q \in \mathcal{D}}}_{Q \subset R^i_j} |A_Q(x)|^2  \Big]^{p/2}\,d\mu(x) \\
&\le  \sum_{j \in \Z} 2^{pj} \sum_i \int_{R^i_j} \Big[ \mathop{\sum_{Q \in \mathcal{D}}}_{Q \subset R^i_j} |A_Q(x)|^2  \Big]^{p/2}\,d\mu(x) \\
&\le\car_p( (A_Q)_{Q \in \mathcal{D}})^p \sum_{j \in \Z} 2^{pj} \mu\Big( \bigcup_i R^i_j \Big) \\
&\le \car_p( (A_Q)_{Q \in \mathcal{D}})^p \sum_{j \in \Z} 2^{pj}  \mu( \{ M_{\mu}^{\mathcal{D}} f > 2^j \}) \\
&\approx \car_p( (A_Q)_{Q \in \mathcal{D}})^p \| M_{\mu}^{\mathcal{D}} f \|_{L^p(\mu)}^p \lesssim \car_p( (A_Q)_{Q \in \mathcal{D}})^p \|f\|_{L^p(\mu)}^p,
\end{align*}
where $M_{\mu}^{\mathcal{D}}$ stands for the dyadic Hardy-Littlewood maximal operator. Here we used the assumption $p \in (1,2]$ simply via the fact that $(a+b)^{\gamma} \le a^{\gamma} + b^{\gamma}$ for $a,b \ge 0$ and $\gamma \in (0, 1]$.
\end{proof}

\subsection{Twisted martingale difference operators and square function estimates}
If $Q \in \D$, $Q\subset Q^*$, and $f \in L^1_{\textup{loc}}(\mu)$, we define the twisted martingale difference operators
\begin{equation*}
\Delta_Q f = \sum_{Q' \in \, \textrm{ch}(Q)} \Big[\frac{\langle f \rangle_{Q'}}{\langle b_{(Q')^a}\rangle_{Q'}}b_{(Q')^a} - \frac{\langle f \rangle_Q}{\langle b_{Q^a}\rangle_Q}b_{Q^a}\Big]1_{Q'}.
\end{equation*}
Note that on the largest $Q^*$ level we agree (by abuse of notation) that $\Delta_{Q^*} = E^b_{Q^*} + \Delta_{Q^*}$, where
$E^b_{Q^*}f = \langle f \rangle_{Q^*} b_{Q^*}$. Therefore, we have that $\int \Delta_Q f \,d\mu = 0$ if $Q \subsetneq Q^*$. We also define 
$$\Delta_k f= \Delta^{Q^*}_k f := \sum_{Q \in \D_k : Q \subset Q^*} \Delta_Q f.$$
Notice that if $\ell(Q^*)=2^s$, then $k \leq s$, that is, only cubes inside the fixed $Q^*$ are considered. 

We now state some lemmata which contain the square function estimates we need in our proof. The first one was proved by Stein on page 103 of \cite{St}: 
\begin{lem}\label{lemma-Stein}
Let $(M, \nu)$ be a $\sigma$-finite measure space and let $\mathfrak{M}$ denote the family of measurable subsets of $M$. Suppose that $\mathcal{F}_1\subseteq \mathcal{F}_2\subseteq \dots$ is an infinite increasing sequence of ($\sigma$-finite) $\sigma$-subalgebras of $\mathfrak{M}$. Let $E_k=E(\cdot | \mathcal{F}_k)$ denote the conditional expectation operator with respect to $\mathcal{F}_k$. Assume that $\{f_k\}_k$ is any sequence of functions on $(M, \nu)$, where $f_k$ is not assumed to be $\mathcal{F}_k$-measurable, and let $(n_k)_k$ be any sequence of positive integers. Then there holds that
\begin{equation}\label{Stein}
\Big\| \Big( \sum_{k \ge 1} |E_{n_k} f_k|^2 \Big)^{1/2} \Big\|_{L^p(\nu)}  \leq A_p \Big\| \Big( \sum_{k \ge 1} |f_k|^2 \Big)^{1/2} \Big\|_{L^p(\nu)}, \quad 1<p<\infty,
\end{equation}
where $A_p$ depends only on $p$. 
\end{lem}
The proof of the next lemma is quite hard. It was proved by Lacey and the first named author \cite{LM1} (but only stated in $L^2(\mu)$). But we will not need the full strength of this, since our function
is bounded. Therefore, instead of using the next lemma, we will indicate a somewhat simpler proof in the $|f| \le 1$ case, which is the only thing we will need. This is not that easy either but
we include the key details for the convenience of the reader.
\begin{lem}\label{lemma.twisted-SF}
Suppose $F \in \mathcal{F}_{Q^*}$ and $f \in L^q(\mu)$. Suppose also that we have constants $\epsilon_Q$, $Q \in \D$, which satisfy $|\epsilon_Q|\leq 1$. Then there holds that 
\begin{equation*}\label{twisted-SF}
\Big\| \mathop{\sum_{Q \in \D}}_{Q^\alpha = F} \epsilon_Q \Delta_Q f \Big\|^q_{L^q(\mu)} \lesssim \|f\|^q_{L^q(\mu)}.
\end{equation*}
\end{lem}
But for us the following consequence is enough (and we will indicate the proof of this simpler statement):

\begin{lem}
Suppose $F \in \mathcal{F}_{Q^*}$ and $|f| \le 1$. Suppose also that we have constants $\epsilon_Q$, $Q \in \D$, which satisfy $|\epsilon_Q|\leq 1$. Then there holds that
\begin{equation}\label{twisted-SF-subset}
\Big\| \mathop{\sum_{Q \in \mathcal{D}}}_{Q^\alpha = F} \epsilon_Q \Delta_Q f \Big\|^q_{L^q(\mu)}  \lesssim \mu(F).
\end{equation}
\end{lem}
\begin{proof}
For the fixed $F \in \mathcal{F}_{Q^*}$, we let $j \in \N$ be such that $F \in \mathcal{F}_{Q^*}^j$ and define $\mathcal{H} = \mathcal{H}_F = \{H \in \mathcal{F}_{Q^*}^{j+1}:\, H \subset F\}$.
For a cube $Q \in \mathcal{D}$ for which $Q^a = F$ we set
\begin{displaymath}
D_Q f := \sum_{Q' \in \textup{ch}(Q) \setminus \mathcal{H}}  \Big[\frac{\langle f \rangle_{Q'}}{\langle b_{F}\rangle_{Q'}}- \frac{\langle f \rangle_Q}{\langle b_{F}\rangle_Q}\Big]1_{Q'}.
\end{displaymath}
The initial step is that
\begin{displaymath}
\Big\| \mathop{\sum_{Q \in \D}}_{Q^\alpha = F} \epsilon_Q \Delta_Q f \Big\|^q_{L^q(\mu)} \lesssim \|f1_F\|_{L^q(\mu)}^q 
+ \Big\| \sup_{\epsilon > 0} \Big| \mathop{\mathop{\sum_{Q \in \mathcal{D}}}_{Q^a = F}}_{\ell(Q) > \epsilon} \epsilon_Q D_Q f \Big|\, \Big\|_{L^q(\mu)}^q.
\end{displaymath}
This works exactly as in \cite{LM1}, proof of Proposition 2.4.

The second step is to show that
\begin{equation}\label{eq:D}
\Big\| \sup_{\epsilon > 0} \Big| \mathop{\mathop{\sum_{Q \in \mathcal{D}}}_{Q^a = F}}_{\ell(Q) > \epsilon} \epsilon_Q D_Q f \Big|\, \Big\|_{L^p(\mu)}^p \lesssim \mu(F), \qquad |f| \le 1,\, p \in (0, \infty).
\end{equation}
The argument we will next give shows that for \eqref{eq:D} it is enough to show that for a fixed $s \in (0, \infty)$ but for all $P \in \mathcal{D}$ there holds that
\begin{equation}\label{eq:JN}
\Big\| \sup_{\epsilon > 0} \Big| \mathop{\mathop{\sum_{Q \in \mathcal{D}:\, Q \subset P}}_{Q^a = F}}_{\ell(Q) > \epsilon} \epsilon_Q D_Q f \Big|\, \Big\|_{L^s(\mu)}^s \le C_1\mu(P).
\end{equation}

Consider a fixed function $f$ for which $|f| \le 1$.
Let us define $\varphi_Q = C_2^{-1}\epsilon_Q D_Q f$, if $Q^a = F$, and $\varphi_Q = 0$ otherwise. Notice that $\|\varphi_Q\|_{L^{\infty}(\mu)} \le 1$ if $C_2 \ge 4$.
Notice also that $\varphi_Q$ is supported on $Q$ and constant on the children $Q' \in \textup{ch}(Q)$. For $P \in \mathcal{D}$ we define
\begin{displaymath}
\Phi_P := \sup_{\epsilon > 0} \Big| \mathop{\mathop{\sum_{Q \in \mathcal{D}}}_{Q \subset P}}_{\ell(Q) > \epsilon} \varphi_Q \Big|
 = C_2^{-1} \sup_{\epsilon > 0} \Big| \mathop{\mathop{\sum_{Q \in \mathcal{D}: \, Q \subset P}}_{Q^a = F}}_{\ell(Q) > \epsilon} \epsilon_Q D_Q f \Big|. 
\end{displaymath}
Suppose we have \eqref{eq:JN} with some $s$ and for all $P$. Then for all $P \in \D$ we have that
\begin{align*}
\mu(\{x \in P: \, \Phi_P(x) > 1\}) \le \int_P \Phi_P^s\,d\mu 
&= C_2^{-s} \Big\| \sup_{\epsilon > 0} \Big| \mathop{\mathop{\sum_{Q \in \mathcal{D}:\, Q \subset P}}_{Q^a = F}}_{\ell(Q) > \epsilon} \epsilon_Q D_Q f \Big|\, \Big\|_{L^s(\mu)}^s \\
&\le C_2^{-s}C_1\mu(P) \\ &\le \mu(P)/2,
\end{align*}
if $C_2 \ge C_1^{1/s}2^{1/s}$. So let us fix $C_2$ large enough.

The non-homogeneous John--Nirenberg principle (see e.g. Lemma 2.8 of \cite{LM1}) now tells us that
for every $P \in \mathcal{D}$ and $t > 1$ there holds that
\begin{displaymath}
\mu(\{x \in P: \Phi_P(x) > t\}) \le 2^{-(t-1)/2} \mu(P).
\end{displaymath}
But then we have for every $p \in (0, \infty)$ and $P \in D$ that
\begin{equation}\label{eq:JN2}
\Big\| \sup_{\epsilon > 0} \Big| \mathop{\mathop{\sum_{Q \in \mathcal{D}:\, Q \subset P}}_{Q^a = F}}_{\ell(Q) > \epsilon} \epsilon_Q D_Q f \Big|\, \Big\|_{L^p(\mu)}^p
\lesssim \int_P \Phi_P^p\,d\mu \lesssim \mu(P).
\end{equation}
With the choice $P = F$ we have \eqref{eq:D}.

So we have reduced to showing \eqref{eq:JN} with some exponent $s \in (0, \infty)$ and for all dyadic cubes $P \in \D$. We will first do this with $f = 1$ and $s = 1/2$, i.e.,
we will prove that for every $P \in \mathcal{D}$ there holds that
\begin{displaymath}
\int_P \Big[\sup_{\epsilon > 0} \Big| \mathop{\mathop{\sum_{Q \in \mathcal{D}}}_{Q^a = F, \, Q \subset P}}_{\ell(Q) > \epsilon} \epsilon_Q D_Q 1 \Big|\Big]^{1/2} \,d\mu \lesssim \mu(P).
\end{displaymath}
Let us write
\begin{displaymath}
\frac {1} {\langle b_F \rangle _{Q'}  } - \frac {1} {\langle b_F \rangle_Q}  =
	\frac {{\langle b_F \rangle_Q}  -  {\langle b_F \rangle _{Q'}  }} {\langle b_F \rangle _{Q} ^2     } +
	\frac { {[\langle b_F \rangle_Q}  -  {\langle b_F \rangle _{Q'}  }]^2 } {{\langle b_F \rangle_Q}^2\langle b_F \rangle _{Q'}     }.
\end{displaymath}
Define $\tilde \epsilon_Q := \epsilon_Q / \langle b_F \rangle _{Q} ^2$, $Q^a  = F$. Note that $|\tilde \epsilon_Q| \lesssim 1$, and then that
\begin{align*}
\int_P&\Big[ \sup_{\epsilon > 0} \Big| \mathop{\mathop{\sum_{Q \in \mathcal{D}}}_{Q^a = F, \, Q \subset P}}_{\ell(Q) > \epsilon} \tilde \epsilon_Q \sum_{Q' \in \textup{ch}(Q) \setminus \mathcal{H}}
[\langle b_F \rangle _{Q'}   - {\langle b_F \rangle_Q}]1_{Q'} \Big| \,d\mu\Big]^{1/2} \\
&\le \mu(P)^{1-1/(2q)} \Big( \int_P \Big[ \sup_{\epsilon > 0} \Big| \mathop{\mathop{\sum_{Q \in \mathcal{D}}}_{Q^a = F, \, Q \subset P}}_{\ell(Q) > \epsilon} \tilde \epsilon_Q \sum_{Q' \in \textup{ch}(Q) \setminus \mathcal{H}}
[\langle b_F \rangle _{Q'}   - {\langle b_F \rangle_Q}]1_{Q'} \Big| \,d\mu\Big]^{q} \Big)^{1/2q} \\
&\le \mu(P)^{1-1/(2q)} \|1_P b_F\|_{L^q(\mu)}^{1/2} \lesssim \mu(P).
\end{align*}
The penultimate estimate follows from Corollary 2.10 of \cite{LM1} (with $p = q$).
For the last inequality we have the following explanation. It is trivial if $F \cap P = \emptyset$ or $F \subset P$. Otherwise,
we may assume that there is a $Q$ for which $Q^a = F$ and $Q \subset P \subset F$. But then $P^a = F$.

The exponent $s = 1/2$ is more useful now when we are dealing with the second term:
\begin{align*}
\int_P& \Big[ \sup_{\epsilon > 0} \Big| \mathop{\mathop{\sum_{Q \in \mathcal{D}}}_{Q^a = F, \, Q \subset P}}_{\ell(Q) > \epsilon} \epsilon_Q \sum_{Q' \in \textup{ch}(Q) \setminus \mathcal{H}}
\frac { {[\langle b_F \rangle_Q}  -  {\langle b_F \rangle _{Q'}  }]^2 } {{\langle b_F \rangle_Q}^2\langle b_F \rangle _{Q'}} 1_{Q'} \Big|\Big]^{1/2} \,d\mu \\
&\lesssim \int_P \Big[ \sum_{Q \in \D} |\Delta_Q^c (1_P b_F) |^2\Big]^{1/2}\,d\mu \\
& \le \mu(P)^{1-1/q} \Big\| \Big[ \sum_{Q \in \D} |\Delta_Q^c (1_P b_F) |^2\Big]^{1/2} \Big\|_{L^q(\mu)} \\ &\lesssim \mu(P)^{1-1/q}\|1_P b_F\|_{L^q(\mu)} \lesssim \mu(P).
\end{align*}
Here
\begin{displaymath}
\Delta_Q^c f = \sum_{Q' \in \textup{ch}(Q)} [\langle f \rangle_{Q'} - \langle f \rangle_Q]1_{Q'}
\end{displaymath}
is the classical martingale difference.
So we have proved \eqref{eq:JN} with $s = 1/2$ and $f=1$ for every $P \in \D$. That means that for $f = 1$ we have \eqref{eq:JN2} with every $p \in (0, \infty)$ and $P \in \D$.

Consider now a function $f$ for which $|f| \le 1$.
Using the above special case we will now prove \eqref{eq:JN} for every $P \in \D$ with $s = 1$.
Let us write
\begin{align} \notag 
	\frac {\langle f \rangle _{Q'}} {\langle b_F \rangle _{Q'}  } 
	-
	\frac {\langle f \rangle _{Q}} {\langle b_F \rangle_Q}  
	&=   
	 \Bigl\{
	\frac {\langle f \rangle _{Q'}} {\langle b_F \rangle _{Q}  } 
	-	\frac {\langle f \rangle _{Q}} {\langle b_F \rangle_Q}  
	\Bigr\}+ 
	\Bigl\{
		\frac {\langle f \rangle _{Q'}} {\langle b_F \rangle _{Q'}  } - 
			\frac {\langle f \rangle _{Q'}} {\langle b_F \rangle _{Q}  } 
	\Bigr\} 
	\\ \label{e:mt1} 
	&= \frac 1 {\langle b_F \rangle _{Q}  } 
	\bigl\{ {\langle f \rangle _{Q'}}  -  {\langle f \rangle _{Q}} 
	\} 
	\\ \label{e:mt2}
	& \qquad +	
	\bigl\{ {\langle f \rangle _{Q'}}  -  {\langle f \rangle _{Q}} \bigr\} 
	\Bigl\{
	\frac 1 { \langle b_F \rangle_{Q'}} - 
	\frac 1 {\langle  b_F\rangle_{Q} } 
	\Bigr\}
	\\ \label{e:mt3}& \qquad + 
	\langle f\rangle_Q 
	\Bigl\{
	\frac 1 { \langle b_F \rangle_{Q'}} - 
	\frac 1 {\langle  b_F\rangle_{Q} } 
	\Bigr\}.
\end{align}
The terms \eqref{e:mt1}-\eqref{e:mt3} give us the corresponding decomposition
\begin{align*}
\epsilon_Q D_Q f = \epsilon_Q^1 \dot \Delta_Q^c f + \epsilon_Q \Delta^{c}_Q f \cdot D_Q 1 + \epsilon_Q^2 D_Q 1,
\end{align*}
where $\Delta_Q^c$ is the classical martingale defined above, $\dot \Delta_Q^c$ is the stopped classical martingale
\begin{displaymath}
\dot \Delta_Q^c f = \sum_{Q' \in \textup{ch}(Q) \setminus \mathcal{H}} [\langle f \rangle_{Q'} - \langle f \rangle_Q]1_{Q'},
\end{displaymath}
and the bounded constants $\epsilon_Q^1$ and $\epsilon_Q^2$ are defined by
\begin{displaymath}
\epsilon_Q^1 =  \frac{\epsilon_Q}{\langle b_F \rangle _{Q}  }, \qquad \epsilon_Q^2 = \epsilon_Q \langle f \rangle_Q.
\end{displaymath}

The first term can be bounded by H\"older (say with $p=2$) and using Corollary 2.10 of \cite{LM1} (with $p = 2$).
The rest exploit the special case $f = 1$. The second term can be bounded by bringing the absolute values in, using H\"older to the sums with $p=2$, and then using
H\"older in the integral with $p=2$. Here one needs \eqref{eq:JN2} with $f=1$ and $p=2$. The last term is just  \eqref{eq:JN2} with $f=1$ and $p=1$. We are done.
\end{proof}

In the $|f| \le 1$ case we can get rid of the assumption $Q^a = F$ as follows:
\begin{prop}\label{MD-square.function.estimate}
Let $|f| \leq 1$. Then there holds that
\begin{equation}\label{eq:MD-square.function.estimate}
\Big\| \Big( \sum_{k} |\Delta_{k} f|^2 \Big)^{1/2} \Big\|^q_{L^q(\mu)} \lesssim \mu(Q^*).
\end{equation}
\end{prop}

\begin{proof}
By Khinchine's inequality there holds that
$$\Big\| \Big( \sum_{k} |\Delta_{k} f|^2 \Big)^{1/2} \Big\|_{L^q(\mu)} \lesssim \Big\|\sum_{k } \eps_k \Delta_{k} f \Big\|_{L^q(\mu \times \prob)},$$
where $(\eps_k)_{k \in \Z}$ is a random sequence of Rademacher functions, i.e.,  a sequence of independent random variables attaining values $\pm 1$ with an equal probability $\prob(\eps_k=1)=\prob(\eps_k=-1)=1/2$. If we  set $\epsilon_Q=\eps_k$, when $Q \in \D_k$, we have that
\begin{align*}
&\Big\| \sum_{k} \eps_k \sum_{Q\in \D_k: Q \subset Q^*} \Delta_Q f\Big\|_{L^q(\mu \times \prob)}=  \Big\|\sum_{Q\in \D: Q \subset Q^*} \epsilon_Q \Delta_Q f\Big\|_{L^q(\mu \times \prob)}\\
&\leq \sum_{j\geq 0}\Big( \sum_{F \in \mathcal{F}^j_{Q^*}} \Big\| \mathop{\sum_{Q\in \D: Q \subset Q^*}}_{Q^\alpha=F}  \epsilon_Q \Delta_Q f \Big\|^q_{L^q(\mu \times \prob)} \Big)^{1/q} \\
&\lesssim   \sum_{j \geq 0} \Big(\sum_{F \in \mathcal{F}^j_{Q^*}} \mu(F) \Big)^{1/q}\lesssim \mu(Q^*)^{1/q},
\end{align*}
where the second-to-last inequality follows from \eqref{twisted-SF-subset} and $\int d\prob=1$, and the last one from \eqref{StoppingFk}.
\end{proof}

\section{Reductions towards the proof of the key inequality}
We will estimate the quantity
\begin{align*}
\Big[ \int_{Q_0} \Big( \int_0^{\ell(Q_0)} |\theta_t f(x)|^2 \frac{dt}{t} \Big)^{q/2} \,d\mu(x)\Big]^{1/q}
\end{align*}
for an arbitrary fixed cube $Q_0 \subset \R^n$ and for an arbitrary fixed function $f$ satisfying that $|f| \le 1_{Q_0}$ (the choice $f = 1_{Q_0}$ would suffice).
Let $s$ be defined by $2^{s-1} \leq \ell(Q_0) < 2^{s}$. 

\subsection{Reduction to a dyadic setting of good geometric data}\label{reduction.good.R}
For a fixed $w \in (\{0,1\}^n)^\Z$ and $x \in Q_0$ we have that
\begin{align*}
 \int_{0}^{\ell(Q_0)} |\theta_{t}f(x)|^{2} \, \frac{dt}{t} \leq \mathop{\sum_{R \in \mathcal{D}(w)}}_{\ell(R) \leq 2^s} 1_R(x) \int_{\ell(R)/2}^{\ell(R)} |\theta_{t}f(x)|^{2} \, \frac{dt}{t}.
\end{align*}
Recall the constants from \eqref{eq:Intro.qt1}. To prove \eqref{Intro:Carleson} 
we note that by above it is enough to prove that
\begin{equation}\label{eq:red.good1}
E_w \Big[ \int_{\R^n} 1_{Q_0}(x)\Big(  \mathop{\sum_{R \in \mathcal{D}(w)}}_{\ell(R) \le 2^s} 1_R(x) \int_{\ell(R)/2}^{\ell(R)} |\theta_{t}f(x)|^{2} \, \frac{dt}{t} \Big)^{q/2}\,d\mu(x)\Big]^{1/q}
\end{equation}
can be bounded by
\begin{displaymath}
[C_3 (1 +V_{\textup{loc}, q}) + C_2^{-1} \|V\|_{L^q(\mu) \to L^q(\mu)}/2]\mu(3Q_0)^{1/q}.
\end{displaymath}

We can estimate the quantity in \eqref{eq:red.good1} by
\begin{align*}
&E_w \Big[ \int_{\R^n} 1_{Q_0}(x)\Big(  \mathop{\sum_{R \in \mathcal{D}(w)_{\textup{good}}}}_{\ell(R) \le 2^s} 1_R(x) \int_{\ell(R)/2}^{\ell(R)} |\theta_{t}f(x)|^{2} \, \frac{dt}{t} \Big)^{q/2}\,d\mu(x)\Big]^{1/q} \\
&+ E_w \Big[ \int_{\R^n} 1_{Q_0}(x) \Big(  \mathop{\sum_{R \in \mathcal{D}(w)_{\textup{bad}}}}_{\ell(R) \le 2^s} 1_R(x) \int_{\ell(R)/2}^{\ell(R)} |\theta_{t}f(x)|^{2} \, \frac{dt}{t} \Big)^{q/2}\,d\mu(x)\Big]^{1/q}.
\end{align*}
Using $E g^{\alpha} \le (Eg)^{\alpha}$ for $\alpha \in (0,1]$, we see (with $\alpha = 1/q$ and $\alpha = q/2$) that
\begin{align*}
&E_w \Big[ \int_{\R^n} 1_{Q_0}(x)\Big(  \mathop{\sum_{R \in \mathcal{D}(w)_{\textup{bad}}}}_{\ell(R) \le 2^s} 1_R(x) \int_{\ell(R)/2}^{\ell(R)} |\theta_{t}f(x)|^{2} \, \frac{dt}{t} \Big)^{q/2}\,d\mu(x)\Big]^{1/q} \\
&\le \Big[ \int_{\R^n} 1_{Q_0}(x) \Big(  E_w \mathop{\sum_{R \in \mathcal{D}(w)_{\textup{bad}}}}_{\ell(R) \le 2^s} 1_R(x) \int_{\ell(R)/2}^{\ell(R)} |\theta_{t}f(x)|^{2} \, \frac{dt}{t} \Big)^{q/2}\,d\mu(x)\Big]^{1/q}.
\end{align*}
Using the fact that $w \mapsto 1_{\textup{bad}}(R_0 + w)$ is independent of $w \mapsto 1_{R_0 + w}(x)$ for every $R_0 \in \mathcal{D}_0$, and that $E_w 1_{\textup{bad}}(R_0 + w) \le c(r) \to 0$ when $r \to \infty$, we have
\begin{align*}
E_w \Big[ \int_{\R^n} 1_{Q_0}(x)& \Big(  \mathop{\sum_{R \in \mathcal{D}(w)_{\textup{bad}}}}_{\ell(R) \le 2^s} 1_R(x) \int_{\ell(R)/2}^{\ell(R)} |\theta_{t}f(x)|^{2} \, \frac{dt}{t} \Big)^{q/2}\,d\mu(x)\Big]^{1/q} \l  \\ &\le c(r)^{1/2}\|Vf\|_{L^q(\mu)} \le 
(2C_2)^{-1} \|V\|_{L^q(\mu) \to L^q(\mu)} \mu(3Q_0)^{1/q}
\end{align*}
fixing $r \lesssim 1$ large enough (note that $c(r) = C(n, \alpha, m)2^{-r\gamma}$).

We have reduced to showing that uniformly on $w \in (\{0,1\}^n)^\Z$ the quantity
\begin{align*}
\Big[ \int_{\R^n} 1_{Q_0}(x) \Big(  \mathop{\sum_{R \in \mathcal{D}(w)_{\textup{good}}}}_{\ell(R) \le 2^s} 1_R(x) \int_{\ell(R)/2}^{\ell(R)} |\theta_{t}f(x)|^{2} \, \frac{dt}{t} \Big)^{q/2}\,d\mu(x)\Big]^{1/q} 
\end{align*}
can be dominated by $C_3(1 + V_{\textup{loc},q}) \mu(3Q_0)^{1/q}$. We fix one $w$ and write $\D = \D(w)$.

\subsection{Decomposition of $f$}\label{reduction.begin}
Since $f \in L^q(\mu)$ is supported in $Q_0$ we may expand
\begin{equation}
f = \mathop{\mathop{\sum_{Q^* \in \mathcal{D}}}_{\ell(Q^*) = 2^s}}_{Q_0 \cap Q^* \ne \emptyset} \mathop{\sum_{Q \in \mathcal{D}}}_{Q \subset Q^*} \Delta_Q f.
\label{eq:md-decomp}\end{equation}
Notice that there are only finitely many such $Q^*$ and always $Q^* \subset 3 Q_0$.
Define
\begin{displaymath}
A_\kappa f(x) :=  \Big(\mathop{\sum_{R \in \mathcal{D}_{\textup{good}}}}_{2^{-\kappa} < \ell(R) \le 2^s} 1_R(x) \int_{\ell(R)/2}^{\ell(R)} |\theta_{t}f(x)|^{2} \, \frac{dt}{t} \Big)^{1/2}
\end{displaymath}
and
\begin{displaymath}
A f(x) :=  \Big(\mathop{\sum_{R \in \mathcal{D}_{\textup{good}}}}_{\ell(R) \le 2^s} 1_R(x) \int_{\ell(R)/2}^{\ell(R)} |\theta_{t}f(x)|^{2} \, \frac{dt}{t} \Big)^{1/2}.
\end{displaymath}
Notice that for $x \in Q_0$ there holds that
\begin{align*}
&\Big|Af(x) - A_\kappa\Big( \sum_{Q^*} \mathop{\sum_{Q \subset Q^*}}_{\ell(Q) > 2^{-\kappa}} \Delta_Q f\Big)(x)\Big| \\
&\le |Af(x) - A_{\kappa}f(x)| +  \Big|A_{\kappa}f(x) - A_{\kappa}\Big( \sum_{Q^*} \mathop{\sum_{Q \subset Q^*}}_{\ell(Q) > 2^{-\kappa}} \Delta_Q f\Big)(x)\Big| \\
&\le  \Big(\mathop{\sum_{R \in \mathcal{D}_{\textup{good}}}}_{\ell(R) \le 2^{-\kappa}} 1_R(x) \int_{\ell(R)/2}^{\ell(R)} |\theta_{t}f(x)|^{2} \, \frac{dt}{t} \Big)^{1/2} + A_{\kappa}\Big( f - \sum_{Q^*} \mathop{\sum_{Q \subset Q^*}}_{\ell(Q) > 2^{-\kappa}} \Delta_Q f\Big)(x) \\
&\le \Big( \int_0^{2^{-\kappa}}  |\theta_{t}f(x)|^{2} \, \frac{dt}{t} \Big)^{1/2} + V\Big( f - \sum_{Q^*} \mathop{\sum_{Q \subset Q^*}}_{\ell(Q) > 2^{-\kappa}} \Delta_Q f\Big)(x).
\end{align*}
It follows by dominated convergence and the fact that $V$ is bounded on $L^q(\mu)$ that
\begin{align*}
\lim_{\kappa \to \infty} \Big\| 1_{Q_0}\Big( Af -  A_{\kappa} \Big( \sum_{Q^*} \mathop{\sum_{Q \subset Q^*}}_{\ell(Q) > 2^{-\kappa}} \Delta_Q f\Big) \Big)\Big\|_{L^q(\mu)} = 0.
\end{align*}
We have reduced to showing that
\begin{equation}\label{heart.estimate}
\Big[ \int_{\R^n} 1_{Q_0}(x) \Big(  \mathop{\sum_{R \in \mathcal{D}_{\textup{good}}}}_{2^{-\kappa} < \ell(R) \le 2^s} 1_R(x) \int_{\ell(R)/2}^{\ell(R)} \Big| \mathop{\mathop{\sum_{Q \in \mathcal{D}}}_{Q \subset Q^*}}_{\ell(Q) > 2^{-\kappa}} \theta_{t} \Delta_Q f(x)\Big|^{2} \, \frac{dt}{t}\Big)^{q/2}\,d\mu(x)\Big]^{1/q}
\end{equation}
can be dominated by $C_3(1 + V_{\textup{loc},q}) \mu(Q^*)^{1/q}$ for every fixed $\kappa$ and for every fixed $Q^*$. We used the fact that 
\begin{displaymath}
\theta_t\Big( \sum_{Q^*} \mathop{\sum_{Q \subset Q^*}}_{\ell(Q) > 2^{-\kappa}} \Delta_Q f \Big) =  \sum_{Q^*} \mathop{\sum_{Q \subset Q^*}}_{\ell(Q) > 2^{-\kappa}} \theta_t \Delta_Q f,
\end{displaymath}
since the sum is finite for every $\kappa$. To fix only one $Q^* \subset 3Q_0$ we used the fact that $\#\{Q^* \in \mathcal{D}:\, \ell(Q^*) = 2^s \textup{ and } Q^* \cap Q_0 \ne \emptyset\} \lesssim 1$.

\subsection{Splitting the summation}
We will split the sum \eqref{heart.estimate} in to the following four pieces:
\begin{itemize}
\item[$Q$:] $\ell(Q)< \ell(R)$;
\item[$Q$:] $\ell(Q)\geq \ell(R)$ and $d(Q,R)> \ell(R)^\gamma \ell(Q)^{1-\gamma}$;
\item[$Q$:] $\ell(R) \leq \ell(Q) \leq 2^r \ell(R)$ and $d(Q,R) \leq \ell(R)^\gamma \ell(Q)^{1-\gamma}$;
\item[$Q$:] $\ell(Q)>2^r\ell(R)$ and $d(Q,R) \leq \ell(R)^\gamma \ell(Q)^{1-\gamma}$.
\end{itemize}
We call the second sum the separated sum, the third sum the diagonal sum and the last sum the nested sum. Thus, \eqref{heart.estimate} is bounded by
$$I_{\ell(Q)<\ell(R)} + I_{\textup{sep}} + I_{\textup{diag}} + I_{\textup{nested}}.$$
We bound these four pieces in the four subsequent chapters.
\begin{rem}
The $\kappa$ and the $s$ are fixed and sometimes such implicit conditions on the generations of the cubes are not written down.
\end{rem}
\section{The case $\ell(Q)< \ell(R)$}\label{section:Q<R}
We start by proving the following lemma.
\begin{lem}\label{lemma.kernel.est.Q<R}
Let $Q, R \in \D$ be such that $\ell(R)/\ell(Q)=2^\ell$ and $D(Q,R)/\ell(R) \sim 2^j$ for $\ell \ge 1$ and $j \ge 0$. Then, if $S_0=Q^{(\ell+j+\theta(j))}$, $ x\in R$ and $y\in Q$, there holds that
\begin{equation}\label{kernel.est.Q<R}
|s_t(x,y)-s_t(x,c_Q)| \lesssim 2^{-\alpha \ell} 2^{-3\alpha j/4} \ell(S_0)^{-m},  \quad  t \in (\ell(R)/2, \ell(R)).
\end{equation}
Here $c_Q$ denotes the centre of $Q$.
\end{lem}
\begin{proof}
First, notice that for every $y\in Q$ we have that $|y-c_Q| \leq \ell(Q)/2 \leq \ell(R)/4 < t/2$. Therefore, we may use \eqref{eq:yhol} to obtain
$$|s_t(x,y)-s_t(x,c_Q)| \lesssim \frac{\ell(Q)^\alpha}{(\ell(R) +d(Q,R))^{m+\alpha}} \lesssim  \frac{\ell(Q)^\alpha}{D(Q,R)^{m+\alpha}},$$
where we used that obviously $D(Q,R) \lesssim \ell(R) +d(Q,R)$ in our situation.
Next, observe that
\begin{displaymath}
\frac{\ell(Q)^\alpha}{D(Q,R)^{m+\alpha}} \approx 2^{-\alpha \ell} 2^{-(m+\alpha)j} \ell(R)^{-m}.
\end{displaymath}
Using the estimate $m \gamma/ (1-\gamma) < \alpha/4$ and the definition of $S_0$ we see that
\begin{displaymath}
\ell(S_0)^{-m} \gtrsim 2^{-mj - \alpha j /4} \ell(R)^{-m}.
\end{displaymath}
Combining we get \eqref{kernel.est.Q<R}.
\end{proof}
Let $Q \in \D$ and $R \in \D_\textup{good}$ be such that $\ell(R)/\ell(Q)=2^\ell$ and $D(Q,R)/\ell(R) \sim 2^j$ for $\ell \ge 1$ and $j \ge 0$.
Assume also that $(x,t) \in W_R$. Since $\ell(Q) < \ell(R) \le 2^s$, we have $\int \Delta_Q f\,d\mu= 0$. Using this we write
$$|\theta_t \Delta_Q f (x)|=\Big| \int_Q [s_t(x,y)-s_t(x,c_Q)] \Delta_Q  f(y)\, d\mu(y)\Big|.$$ 
Using the estimate \eqref{kernel.est.Q<R} we now see that
\begin{displaymath}
|\theta_t \Delta_Q f (x)| \lesssim 2^{-\alpha \ell} 2^{-3\alpha j/4} \ell(S_0)^{-m} \int_Q |\Delta_Q f(y)| d\mu(y),
\end{displaymath}
where $Q, R \subset S_0 :=Q^{(\ell+j+\theta(j))}$ (by (1) of Lemma \ref{lemma.R<S(Q)}).

We can now see that $I_{\ell(Q) < \ell(R)}$ can be dominated by
\begin{equation*}
\sum_{j, \ell}2^{-\frac{\alpha}{2}(\ell+\frac{3}{4}j)} \Big\|  \Big(\sum_{k \le s} \sum_{R \in \mathcal{D}_{k, \textup{good}}}1_{R} \Big( \mathop{\sum_{Q \in \mathcal{D}_{k-\ell}:\, Q \subset Q^*}}_{D(Q,R)/\ell(R)\sim 2^j }  \ell(S_0)^{-m} \int |\Delta_Q f| d\mu \Big)^2\Big)^{1/2} \Big\|_{L^q(\mu)}.
\end{equation*}
Let us fix $j, \ell, k$. Set $\tau_j(k):= j + \theta(j) + k = \textup{gen}(S_0)$.
We have by disjointness considerations and the fact that $Q, R \subset S_0$ that
\begin{align*}
\sum_{R \in \mathcal{D}_{k, \textup{good}}} &1_{R} \Big( \mathop{\sum_{Q \in \mathcal{D}_{k-\ell}:\, Q \subset Q^*}}_{D(Q,R)/\ell(R)\sim 2^j } \ell(S_0)^{-m} \int |\Delta_Q f|\, d\mu \Big)^2\\
&=\Big( \sum_{R \in \mathcal{D}_{k, \textup{good}}}1_{R} \mathop{\sum_{Q \in \mathcal{D}_{k-\ell}:\, Q \subset Q^*}}_{D(Q,R)/\ell(R)\sim 2^j } 2^{-m \tau_j(k)} \int |\Delta_Q f|\, d\mu \Big)^2\\
&=\Big( \sum_{S \in \D_{\tau_j(k)}}\mathop{\sum_{R \in \mathcal{D}_{k, \textup{good}}}}_{R\subset S}1_{R} \mathop{\sum_{Q \in \mathcal{D}_{k-\ell}:\, Q \subset Q^*}}_{D(Q,R)/\ell(R)\sim 2^j } 2^{-m \tau_j(k)}\int |\Delta_Q f|\, d\mu \Big)^2 \\
&\lesssim \Big( \sum_{S \in \D_{\tau_j(k)}} \frac{1_S}{\mu(S)} \int_S |\Delta_{k-\ell} f| d\mu \Big)^2 \\ &= [E_{\tau_j(k)}(|\Delta_{k-\ell} f|)]^2.
\end{align*}

Note that for fixed $j, \ell$ there holds by Stein's inequality (Lemma \ref{lemma-Stein}) and estimate \eqref{eq:MD-square.function.estimate} that
\begin{align*}
\Big\|  \Big(\sum_{k \le s} [E_{\tau_j(k)}(|\Delta_{k-\ell} f|)]^2 \Big)^{1/2}\Big\|_{L^q(\mu)}
&\lesssim \Big\|  \Big(\sum_{k \le s} |\Delta_k f|^2 \Big)^{1/2}\Big\|_{L^q(\mu)} 
\lesssim \mu(Q^*)^{1/q}.
\end{align*}
We may now conclude that $I_{\ell(Q) < \ell(R)} \lesssim \mu(Q^*)^{1/q}$.

\section{The separated sum}\label{section:separated}
We first prove the following lemma.
\begin{lem}\label{lemma.kernel.est.separ}
Let $Q, R \in \D$ be such that $d(Q,R) > \ell(R)^{\gamma}\ell(Q)^{1-\gamma}$,
$\ell(Q)/\ell(R)=2^\ell$ and $D(Q,R)/\ell(Q) \sim 2^j $ for $\ell, j \geq 0$.
Then, if $S_0=Q^{(j+\theta(j+\ell))}$, $x\in R$ and $y\in Q$, there holds that
\begin{equation}\label{kernel.est.separ}
|s_t(x,y)| \lesssim 2^{-\alpha \ell/4} 2^{- 3\alpha j/4} \ell(S_0)^{-m},\quad  t \in (\ell(R)/2, \ell(R)).
\end{equation}
\end{lem}
\begin{proof}
We begin by noting that
\begin{align*}
|s_t(x,y)| \lesssim \frac{\ell(R)^{\alpha}}{d(Q,R)^{m+\alpha}} \lesssim \frac{\ell(Q)^{\alpha/2}\ell(R)^{\alpha/2}}{D(Q,R)^{m+\alpha}} 
= 2^{-\alpha \ell/2} 2^{-(m+\alpha)j} \ell(Q)^{-m}.
\end{align*}
The second estimate is a standard fact and follows since $(m+\alpha)\gamma \leq \alpha/2$, $\ell(R) \le \ell(Q)$ and $d(Q,R) > \ell(R)^{\gamma}\ell(Q)^{1-\gamma}$.

On the other hand it is easy to see that
\begin{displaymath}
\ell(S_0)^{-m} \gtrsim 2^{-mj - (\ell+j)\alpha/4}\ell(Q)^{-m}.
\end{displaymath}
This uses just the definition of $S_0$ and the bound  $m \gamma/ (1-\gamma) < \alpha/4$.
Combining the estimates we have \eqref{kernel.est.separ}.
\end{proof}

Let $Q \in \D$, $R \in \D_{\textup{good}}$ be such that $d(Q,R) > \ell(R)^{\gamma}\ell(Q)^{1-\gamma}$,
$\ell(Q)/\ell(R)=2^\ell$ and $D(Q,R)/\ell(Q) \sim 2^j $ for $\ell, j \geq 0$. If $(x,t) \in W_R$ we have by \eqref{kernel.est.separ} that
$$|\theta_t \Delta_Q f (x)|\lesssim 2^{-(\ell+j)\alpha/4}\ell(S_0)^{-m} \int |\Delta_Q f(y)|\,d\mu(y),$$ 
where $Q, R \subset S_0 :=Q^{(j+\theta(j+\ell))}$ (by (2) of Lemma \ref{lemma.R<S(Q)}).

We may deduce that $I_{\textup{sep}}$ can be dominated by
\begin{equation*}
\sum_{j, \ell}2^{-\alpha(\ell+j)/4} \Big\| \Big(\sum_{k \le s} \sum_{R \in \mathcal{D}_{k, \textup{good}}} 1_{R} \Big( \mathop{\sum_{Q \in \mathcal{D}_{k+\ell}: Q \subset Q^*}}_{D(Q,R)\sim 2^j \ell(Q)}  \ell(S_0)^{-m}\int |\Delta_Q f|\,d\mu \Big)^2\Big)^{1/2} \Big\|_{L^q(\mu)}.
\end{equation*}
A completely analogous estimate to that of the previous section shows that $I_{\textup{sep}} \lesssim \mu(Q^*)^{1/q}$.

\section{The diagonal sum}\label{section:diagonal}
Let $Q \in \D$, $R \in \D_{\textup{good}}$ be such that $\ell(Q) / \ell(R) = 2^{\ell}$ and $D(Q,R) / \ell(Q) \sim 2^j$. Since we are in the diagonal summation $I_{\textup{diag}}$ we have that
$\ell, j \lesssim 1$. If $(x,t) \in W_R$ we have that
\begin{displaymath}
|s_t(x,y)| \lesssim t^{-m} \approx \ell(R)^{-m} \approx \ell(S_0)^{-m},
\end{displaymath}
where $Q, R \subset S_0 :=Q^{(j+\theta(j+\ell))}$ (by (2) of Lemma \ref{lemma.R<S(Q)}). It is now clear by the previous arguments that
$I_{\textup{diag}} \lesssim \mu(Q^*)^{1/q}$.

\section{The nested sum}\label{section:nested}
In this case one uses the goodness of $R$ to conclude that one must actually have that $R \subset Q$.
Therefore, things reduce to proving that
\begin{equation}\label{Inested}
 \Big\| 1_{Q_0} \Big( \mathop{\sum_{R \in \mathcal{D}_{\textup{good}}:\, R \subset Q^*}}_{2^{-\kappa} < \ell(R) < 2^{s-r}} 
 1_{R} \int_{\ell(R)/2}^{\ell(R)} \Big| \mathop{\sum_{\ell=r+1}^{s-\textup{gen(R)}}} \theta_t \Delta_{R^{(\ell)}}  f \Big|^2\,
 \frac{dt}{t}\Big)^{1/2} \Big\|_{L^q(\mu)} \lesssim \mu(Q^*)^{1/q}.
\end{equation}
We bound the right hand side of \eqref{Inested} by $I_{\textup{nested}, 1} + I_{\textup{nested}, 2}$, where
\begin{align*}
I_{\textup{nested}, 1} = \Big\| 1_{Q_0} \Big( \mathop{\sum_{R \in \mathcal{D}_{\textup{good}}:\, R \subset Q^*}}_{2^{-\kappa} < \ell(R) < 2^{s-r}} 
 1_{R} \int_{\ell(R)/2}^{\ell(R)} \Big| \mathop{\sum_{\ell=r+1}^{s-\textup{gen(R)}}} \theta_t (1_{R^{(\ell)} \setminus R^{(\ell-1)}}\Delta_{R^{(\ell)}}  f) \Big|^2\,
 \frac{dt}{t}\Big)^{1/2} \Big\|_{L^q(\mu)} 
\end{align*}
and
\begin{align*}
I_{\textup{nested}, 2} = \Big\| 1_{Q_0} \Big( \mathop{\sum_{R \in \mathcal{D}_{\textup{good}}:\, R \subset Q^*}}_{2^{-\kappa} < \ell(R) < 2^{s-r}} 
 1_{R} \int_{\ell(R)/2}^{\ell(R)} \Big| \mathop{\sum_{\ell=r+1}^{s-\textup{gen(R)}}} \theta_t (1_{R^{(\ell-1)}}\Delta_{R^{(\ell)}}  f) \Big|^2\,
 \frac{dt}{t}\Big)^{1/2} \Big\|_{L^q(\mu)}.
\end{align*}

\subsection{The sum $I_{\textup{nested}, 1}$}
The following lemma is the key to handling this sum.
\begin{lem}
For $\ell \ge r+1$ and $R \in \mathcal{D}_{k, \textup{good}}$ we have for $(x,t) \in W_R$ that there holds that
\begin{displaymath}
|\theta_t(1_{R^{(\ell)} \setminus R^{(\ell-1)}} \Delta_{R^{(\ell)}}f)(x)| \lesssim 2^{-\alpha \ell/2} 2^{-(k+\ell)m} \int |\Delta_{R^{(\ell)}} f(y)| \,d\mu(y).
\end{displaymath}
\end{lem}
\begin{proof}
If $S \in \textup{ch}(R^{(\ell)})$, $S \ne R^{(\ell-1)}$, and $(x,t) \in W_R$, we have by the size estimate \eqref{eq:size} that
\begin{align*}
|\theta_t(1_S \Delta_{R^{(\ell)}}f)(x)| &\lesssim \int_S \frac{\ell(R)^{\alpha}}{d(S,R)^{m+\alpha}} |\Delta_{R^{(\ell)}} f(y)|\,d\mu(y) \\
&\lesssim \int_S \Big(\frac{\ell(R)}{\ell(S)}\Big)^{\alpha/2}\frac{1}{\ell(S)^m} |\Delta_{R^{(\ell)}} f(y)|\,d\mu(y)
\end{align*}
Here we used that by goodness $d(R, S) \ge \ell(R)^{\gamma}\ell(S)^{1-\gamma}$, and that we have $\gamma \leq \alpha/(2m+2\alpha)$.
\end{proof}
We now see using this lemma that $I_{\textup{nested}, 1}$ can dominated by
\begin{align*}
& \sum_{\ell\geq r+1} 2^{-\alpha \ell/2} \Big\| 1_{Q_0} \Big( \sum_{k\le s-\ell} \Big(\sum_{R \in \mathcal{D}_{k,\textup{good}}:\, R \subset Q^*}
1_{R}  2^{-(k+\ell)m} \int |\Delta_{R^{(\ell)}} f(y)| \,d\mu(y)\Big)^2\Big)^{1/2} \Big\|_{L^q(\mu)} \\
&\lesssim \sum_{\ell\geq r+1} 2^{-\alpha \ell/2} \Big\| 1_{Q_0} \Big( \sum_{k\le s-\ell} \Big(\sum_{S \in \mathcal{D}_{k+\ell}:\, S \subset Q^*} \frac{1_S}{\mu(S)}
 \int |\Delta_{S} f(y)| \,d\mu(y)\Big)^2\Big)^{1/2} \Big\|_{L^q(\mu)} \\
 &= \sum_{\ell\geq r+1} 2^{-\alpha \ell/2} \Big\| 1_{Q_0} \Big( \sum_{k\le s-\ell} \Big(\sum_{S \in \mathcal{D}_{k+\ell}} \frac{1_S}{\mu(S)}
 \int_S |\Delta_{k+\ell} f(y)| \,d\mu(y)\Big)^2\Big)^{1/2} \Big\|_{L^q(\mu)} \\
 &\lesssim  \Big\|  \Big( \sum_{k \le s} [ E_{k} (|\Delta_{k}f|)]^2\Big)^{1/2} \Big\|_{L^q(\mu)} \lesssim \mu(Q^*)^{1/q}.
\end{align*}
The last inequality follows from Stein's inequality \eqref{Stein} and \eqref{eq:MD-square.function.estimate}.

\subsection{The sum $I_{\textup{nested}, 2}$}
We begin by recording the following bound:
\begin{lem}\label{lem:kk}
For $\ell \ge r+1$ and $R \in \mathcal{D}_{\textup{good}}$ we have for $(x,t) \in W_R$ that there holds that
\begin{displaymath}
|\theta_t(1_{(R^{(\ell-1)})^c}b_{(R^{(\ell)})^a} )(x)| \lesssim 2^{-\alpha \ell/2}.
\end{displaymath}
\end{lem}
\begin{proof}
Choose $N_0$ so that $(R^{(\ell)})^a=R^{(\ell+N_0)}$. Notice that since $R$ is good there holds that
\begin{align*}
d\big(R, \partial R^{(\ell+j-1)}\big)^{m+\alpha} &\gtrsim 2^{\alpha \ell/2}2^{\alpha j/2} \ell(R)^{\alpha} \mu(R^{(\ell+j)}).
\end{align*}
Here we used that $\gamma(\alpha+m)\leq \alpha/2$. 

Therefore, for $(x,t) \in W_R$, the above estimate, the size bound \eqref{eq:size} and the stopping conditions show that
\begin{align*}
|\theta_t(1_{(R^{(\ell-1)})^c} b_{(R^{(\ell)})^a }(x)| 
&\lesssim \sum_{j=0}^{N_0}  \int_{R^{(\ell+j)} \setminus (R^{(\ell+j-1)} } \frac{\ell(R)^\alpha}{|x-y|^{m+\alpha}} |b_{(R^{(\ell)})^a}(y)|\, d\mu(y)\\
&\lesssim  \sum_{j=0}^{N_0}  \frac{\ell(R)^{\alpha} \mu(R^{(\ell+j)})}{2^{\alpha \ell/2}2^{\alpha j/2} \ell(R)^{\alpha} \mu(R^{(\ell+j)})} \\ 
& \lesssim 2^{-\alpha \ell /2}.
\end{align*}
\end{proof}
We now have to do a case study.
\subsection{The case $(R^{(\ell-1)})^a = (R^{(\ell)})^a$}
In this case we may write
\begin{align}\label{eq:split1}
1_{R^{(\ell-1)}} \Delta_{R^{(\ell)}} f = - 1_{(R^{(\ell-1)})^c} B_{R^{(\ell-1)}} b_{(R^{(\ell)})^a} + B_{R^{(\ell-1)}}b_{(R^{(\ell)})^a},
\end{align}
where
\begin{displaymath}
B_{R^{(\ell -1)}} = \frac{\langle f \rangle_{R^{(\ell -1)}}}{\langle b_{(R^{(\ell -1)})^a}\rangle_{R^{(\ell -1)}}} - \frac{\langle f \rangle_{R^{(\ell)}} }{\langle b_{(R^{(\ell)})^a}\rangle_{R^{(\ell)} }}
\end{displaymath}
with the minus term missing if $\ell(R^{(\ell)}) = 2^s$.

Accretivity condition gives that
\begin{align*}
|B_{R^{(\ell-1)}}| \mu(R^{(\ell-1)}) \lesssim \Big| \int_{R^{(\ell-1)}} B_{R^{(\ell-1)}} b_{(R^{(\ell)})^a}\,d\mu\Big| &= \Big| \int_{R^{(\ell-1)}} \Delta_{R^{(\ell)}} f\,d\mu\Big|.
\end{align*}
Combining with Lemma \ref{lem:kk} we see that for $(x,t) \in W_R$ there holds that
\begin{displaymath}
|\theta_t(1_{(R^{(\ell-1)})^c} B_{R^{(\ell-1)}} b_{(R^{(\ell)})^a})(x)| \lesssim 2^{-\alpha \ell/2}   \frac{1}{\mu(R^{(\ell-1)})} \int_{R^{(\ell-1)}} |\Delta_{R^{(\ell)}} f|\,d\mu.
\end{displaymath}
So to control the sum with the first term of \eqref{eq:split1} it is enough to note that for a fixed $\ell \ge r+1$ there holds that
\begin{align*}
&\Big\| 1_{Q_0} \Big( \sum_{k \le s-\ell} \Big( \sum_{R \in \D_k:\, R \subset Q^*}  \frac{1_R}{\mu(R^{(\ell-1)})} \int_{R^{(\ell-1)}} |\Delta_{R^{(\ell)}} f(y)|\,d\mu(y) \Big)^2\Big)^{1/2} \Big\|_{L^q(\mu)} \\
& \le \Big\| 1_{Q_0} \Big( \sum_{k \le s-\ell} \Big( \sum_{S \in \D_{k+\ell-1}}  \frac{1_S}{\mu(S)} \int_{S} |\Delta_{k+\ell} f(y)|\,d\mu(y) \Big)^2\Big)^{1/2} \Big\|_{L^q(\mu)} \\
&= \Big\|  \Big( \sum_{k \le s} [ E_{k-1} (|\Delta_{k}f|)]^2\Big)^{1/2} \Big\|_{L^q(\mu)} \lesssim \mu(Q^*)^{1/q}.
\end{align*}
In the last step we again used Stein's inequality \eqref{Stein} and \eqref{eq:MD-square.function.estimate}. We will not touch the second term of \eqref{eq:split1} yet -- it will become
part of the paraproduct.

\subsection{The case $(R^{(\ell-1)})^a = R^{(\ell-1)}$} 
We decompose
\begin{align*}
 1_{R^{(\ell-1)}} \Delta_{R^{(\ell)}}f = \Big(\frac{\langle f \rangle_{R^{(\ell-1)}}}{\langle b_{R^{(\ell-1)}}\rangle_{R^{(\ell-1)}}}&b_{R^{(\ell-1)}} - \frac{\langle f \rangle_{R^{(\ell)}} }{\langle b_{(R^{(\ell)})^a}\rangle_{R^{(\ell)} }}b_{(R^{(\ell)})^a} \Big) \\
 &+ 1_{(R^{(\ell-1)})^c} \frac{\langle f \rangle_{R^{(\ell)}} }{\langle b_{(R^{(\ell)})^a}\rangle_{R^{(\ell)} }}b_{(R^{(\ell)})^a}.
\end{align*}
The term in the parenthesis will become part of the paraproduct, and we do not touch it further in this subsection.

For the second term, using the construction of the stopping time and Lemma \eqref{lem:kk}, we have for $(x,t) \in W_R$ that
\begin{displaymath}
\Big|\theta_t \Big( 1_{(R^{(\ell-1)})^c} \frac{\langle f \rangle_{R^{(\ell)}} }{\langle b_{(R^{(\ell)})^a}\rangle_{R^{(\ell)} }}b_{(R^{(\ell)})^a} \Big)(x)\Big| \lesssim 2^{-\alpha \ell/2} |\langle f \rangle_{R^{(\ell)}}|.
\end{displaymath}
We say that $R \in \mathcal{S}_\ell$, if $(R^{(\ell-1)})^a = R^{(\ell-1)}$. To control the corresponding sum we note that
\begin{align*}
 \Big\| 1_{Q_0} &\Big( \mathop{\sum_{R \in \mathcal{D}_{\textup{good}}:\, R \subset Q^*}}_{2^{-\kappa} <\ell(R)<2^{s-r}}1_{R}\int_{\ell(R)/2}^{\ell(R)} \Big| \mathop{\sum_{\ell=r+1}^{s-\textup{gen(R)}}}_{R \in \mathcal{S}_\ell}
 \frac{\langle f \rangle_{R^{(\ell)}} }{\langle b_{(R^{(\ell)})^a}\rangle_{R^{(\ell)} }} \theta_t( 1_{(R^{(\ell-1)})^c} b_{(R^{(\ell)})^a}) \Big|^2\frac{dt}{t}\Big)^{1/2} \Big\|_{L^q(\mu)}\\
&\lesssim \Big\| 1_{Q_0} \Big( \mathop{\sum_{R \in \mathcal{D}_{\textup{good}}:\, R \subset Q^*}}_{2^{-\kappa} <\ell(R)<2^{s-r}}1_{R}  \mathop{\sum_{\ell=r+1}^{s-\textup{gen(R)}}}_{R \in \mathcal{S}_\ell}2^{-\alpha \ell/2}  |\langle f \rangle_{R^{(\ell)}}|^2 \Big)^{1/2} \Big\|_{L^q(\mu)}\\
&\leq \Big\| 1_{Q_0} \Big( \sum_{\ell \ge r+1} 2^{-\alpha \ell /2} \sum_{S \in \D:\, S \subset Q^*}  |\langle f \rangle_{S}|^2 A_S \Big)^{1/2}
\Big\|_{L^q(\mu)} \\
&\lesssim \Big\| 1_{Q_0} \Big( \sum_{S \in \D:\, S \subset Q^*}  |\langle f \rangle_{S}|^2 A_S \Big)^{1/2}\Big\|_{L^q(\mu)} \\
&  \lesssim \mu(Q^*)^{1/q},
\end{align*}
where we denoted
\begin{displaymath}
A_S(x) := \mathop{\sum_{S' \in \textup{ch}(S)}}_{(S')^a = S'} 1_{S'}(x).
\end{displaymath}
For the final estimate one can use the fact that $|f| \le 1$ to throw away the averages, and then use H\"older with exponent $p:= 2/q > 1$ together with Lemma \ref{lem:mescar}:
\begin{align*}
\Big\|& \Big( \sum_{S \in \D:\, S \subset Q^*}  A_S \Big)^{1/2}\Big\|_{L^q(\mu)} \\
& \le \Big[ \mu(Q^*)^{1-1/p} \Big( \sum_{F \in \mathcal{F}_{Q^*}} \mu(F) \Big)^{1/p} \Big]^{1/q} \\
&\lesssim ( \mu(Q^*)^{1-1/p} \mu(Q^*)^{1/p} )^{1/q}\\
& = \mu(Q^*)^{1/q}.
\end{align*}
\subsection{The Carleson estimate for the paraproduct}
Combining the above two cases and collapsing the remaining telescoping summation we are left with: 
\begin{align*}
\Big\| 1_{Q_0} \Big(& \mathop{\sum_{R \in \mathcal{D}_{\textup{good}}:\, R \subset Q^*}}_{2^{-\kappa} <\ell(R)<2^{s-r}}1_{R} \int_{\ell(R)/2}^{\ell(R)} \Big|  \frac{\langle f \rangle_{R^{(r)}}}{\langle b_{(R^{(r)})^a} \rangle_{R^{(r)}}} \theta_t b_{(R^{(r)})^a}\Big|^2\,\frac{dt}{t}\Big)^{1/2} \Big\|_{L^q(\mu)} \\
&\lesssim \Big\| 1_{Q_0} \Big( \mathop{\sum_{S\in \D}}_{S\subset Q^*} \mathop{\sum_{R \in \D }}_{R^{(r)}=S} 1_{R} \int_{\ell(R)/2}^{\ell(R)} \big| \theta_t b_{S^a}\big|^2\,\frac{dt}{t}\Big)^{1/2} \Big\|_{L^q(\mu)}\\
& = \Big\| 1_{Q_0} \Big( \sum_{F \in \mathcal{F}_{Q^*}} \sum_{S: S^\alpha=F} \mathop{\sum_{R \in \D }}_{R^{(r)}=S} 1_{R} \int_{\ell(R)/2}^{\ell(R)} \big| \theta_t b_{F}\big|^2\,\frac{dt}{t}\Big)^{1/2} \Big\|_{L^q(\mu)}\\
&  \leq \Big\| 1_{Q_0} \Big( \sum_{F \in \mathcal{F}_{Q^*}} \sum_{R: R \subset F}  1_{R} \int_{\ell(R)/2}^{\ell(R)} \big| \theta_t b_{F}\big|^2\,\frac{dt}{t}\Big)^{1/2} \Big\|_{L^q(\mu)}\\
&  \leq \Big\| 1_{Q_0} \Big( \sum_{F \in \mathcal{F}_{Q^*}} 1_{F} \int_{0}^{\ell(F)} \big| \theta_t b_{F}\big|^2\,\frac{dt}{t} \Big)^{1/2} \Big\|_{L^q(\mu)}\\
&  \leq \sum_{j\geq 0}  \Big( \sum_{F \in \mathcal{F}^j_{Q^*}} \Big\|  \Big(1_{F} \int_{0}^{\ell(F)} \big| \theta_t b_{F}\big|^2\,\frac{dt}{t}\Big)^{1/2} \Big\|^q_{L^q(\mu)} \Big)^{1/q}\\
& \le  V_{\textup{loc}, q} \sum_{j\geq 0} \Big(  \sum_{F \in \mathcal{F}^j_{Q^*}} \mu(F)\Big)^{1/q} \lesssim V_{\textup{loc}, q} \mu(Q^*)^{1/q}.
\end{align*}
In the first inequality we used the stopping time conditions and the fact that $|f|\leq 1$, while the penultimate inequality follows from assumption $(4)$ of theorem \ref{thm:main}.

\newpage
\appendix
\section{$T1$ theorem in $L^q(\mu)$} \label{App:T1}
Let us recall the definition of our Carleson constant:
\begin{displaymath}
\mcar_V(q, \lambda) :=  \mathop{\sup_{Q_0 \subset \R^n}}_{\textup{cube}} \Big[ \frac{1}{\mu(\lambda Q_0)} \int_{Q_0} \Big( \int_0^{\ell(Q_0)} |\theta_t 1(x)|^2 \frac{dt}{t} \Big)^{q/2} \,d\mu(x)\Big]^{1/q}.
\end{displaymath}
Recall also that $q \in (1,2]$.
We are interested in proving the following $T1$ theorem.
\begin{thm}\label{appendix:T1thm}
We have the quantitative bound
\begin{equation}\label{eq:qt1}
\|V\|_{L^q(\mu) \to L^q(\mu)} \lesssim 1 + \mcar_V(q, 9).
\end{equation}
\end{thm}
We now indicate the proof of this theorem. We can again, without loss of generality, assume that $\|V\|_{L^q(\mu) \to L^q(\mu)} \ < \infty$.

\subsection{Reduction to a dyadic setting of good geometric data}
Since we are not so well localised yet this part of the argument has a few more steps than that of the main theorem.
We write
\begin{align*}
 \int_{0}^{\infty} |\theta_{t}f(x)|^{2} \, \frac{dt}{t} &= \sum_{R \in \mathcal{D}(w)} 1_R(x) \int_{\ell(R)/2}^{\ell(R)} |\theta_{t}f(x)|^{2} \, \frac{dt}{t} \\
 &= \lim_{s \to \infty} \mathop{\sum_{R \in \mathcal{D}(w)}}_{\ell(R) \le 2^s} 1_R(x) \int_{\ell(R)/2}^{\ell(R)} |\theta_{t}f(x)|^{2} \, \frac{dt}{t}.
\end{align*}

By monotone convergence we have that
\begin{displaymath}
\|Vf\|_{L^q(\mu)} = \lim_{s \to \infty} \Big[ \int_{\R^n} \Big(  \mathop{\sum_{R \in \mathcal{D}(w)}}_{\ell(R) \le 2^s} 1_R(x) \int_{\ell(R)/2}^{\ell(R)} |\theta_{t}f(x)|^{2} \, \frac{dt}{t} \Big)^{q/2}\,d\mu(x)\Big]^{1/q}.
\end{displaymath}
We take the expectation $E_w$ of this identity. Notice that there holds that
\begin{displaymath}
 \Big[ \int_{\R^n} \Big(  \mathop{\sum_{R \in \mathcal{D}(w)}}_{\ell(R) \le 2^s} 1_R(x) \int_{\ell(R)/2}^{\ell(R)} |\theta_{t}f(x)|^{2} \, \frac{dt}{t} \Big)^{q/2}\,d\mu(x)\Big]^{1/q} \le \|Vf\|_{L^q(\mu)} \in L^1((\{0,1\}^n)^\Z).
\end{displaymath}
Indeed, $\|Vf\|_{L^q(\mu)} < \infty$ and $E_w 1 = 1$. Therefore we have by dominated convergence that
\begin{displaymath}
\|Vf\|_{L^q(\mu)} = \lim_{s \to \infty} E_w \Big[ \int_{\R^n} \Big(  \mathop{\sum_{R \in \mathcal{D}(w)}}_{\ell(R) \le 2^s} 1_R(x) \int_{\ell(R)/2}^{\ell(R)} |\theta_{t}f(x)|^{2} \, \frac{dt}{t} \Big)^{q/2}\,d\mu(x)\Big]^{1/q}.
\end{displaymath}

We now write
\begin{displaymath}
\|V\|_{L^q(\mu) \to L^q(\mu)} = \mathop{\sup_{f \textup{ compactly supported}}}_{\|f\|_{L^q(\mu)} \le 1} \|Vf\|_{L^q(\mu)}.
\end{displaymath}
Fix such $f$, and then fix $N$ so that spt$\,f \subset B(0, 2^N)$. It is enough to prove that for every $s \ge N$ there holds that
\begin{align*}
E_w \Big[ \int_{\R^n} \Big(  \mathop{\sum_{R \in \mathcal{D}(w)}}_{\ell(R) \le 2^s} 1_R(x) &\int_{\ell(R)/2}^{\ell(R)} |\theta_{t}f(x)|^{2} \, \frac{dt}{t} \Big)^{q/2}\,d\mu(x)\Big]^{1/q} \\ &\le C(1 + \mcar_V(q, 9)) + \|V\|_{L^q(\mu) \to L^q(\mu)}/2.
\end{align*}
Now also fix $s \ge N$. One may argue as in \ref{reduction.good.R} and reduce to showing that uniformly on $w \in \Omega$ there holds that
\begin{align*}
\Big[ \int_{\R^n} \Big(  \mathop{\sum_{R \in \mathcal{D}(w)_{\textup{good}}}}_{\ell(R) \le 2^s} 1_R(x) \int_{\ell(R)/2}^{\ell(R)} |\theta_{t}f(x)|^{2} \, \frac{dt}{t} \Big)^{q/2}\,d\mu(x)\Big]^{1/q}
 \lesssim (1 + \mcar_V(q, 9)).
\end{align*}
We fix $w$ and write $\D = \D(w)$.

\subsection{Expanding $f$ and splitting the summation}
We now expand the fixed $f$ in $L^q(\mu)$ as follows:
\begin{equation}\label{eq:expand}
f = \lim_{\kappa \to \infty} \mathop{\mathop{\sum_{Q^* \in \mathcal{D}}}_{\ell(Q^*) = 2^s}}_{Q^* \cap B(0, 2^N) \ne \emptyset} \mathop{\mathop{\sum_{Q \in \mathcal{D}}}_{Q \subset Q^*}}_{\ell(Q) > 2^{-\kappa}} \Delta_Q f.
\end{equation}
This time the martingales are simple: $\Delta_Q f = \sum_{Q' \in \textup{ch}(Q)} [\langle f \rangle_{Q'} - \langle f \rangle_{Q}]1_{Q'}$ with the understanding that
$\Delta_{Q^*} f =  \sum_{Q' \in \textup{ch}(Q^*)}  \langle f \rangle_{Q'}1_{Q'}$, $\ell(Q^*) = 2^s$. 
The argument in \ref{reduction.begin} shows that it is enough to be able to bound the quantity
\begin{align}
\Big[ \int_{\R^n} \Big(  \mathop{\sum_{R \in \mathcal{D}_{\textup{good}}}}_{2^{-\kappa} < \ell(R) \le 2^s} 1_R(x) \int_{\ell(R)/2}^{\ell(R)} \Big| \mathop{\mathop{\sum_{Q \in \mathcal{D}}}_{Q \subset Q^*}}_{\ell(Q) > 2^{-\kappa}} \theta_{t} \Delta_Q f(x)\Big|^{2} \, \frac{dt}{t}\Big)^{q/2}\,&d\mu(x)\Big]^{1/q} \label{eq:appendix-HEART_EST}
\end{align}
with $C(1 + \mcar_V(q, 9))$ for every fixed $\kappa$ and for every fixed $Q^*$.

The splitting of the summation is the same as in the proof of the main theorem:
the quantity in \eqref{eq:appendix-HEART_EST} is dominated by $I_{\ell(Q)<\ell(R)} + I_{\textup{sep}} + I_{\textup{diag}} + I_{\textup{nested}}$.
The first three terms are treated using similar arguments to the corresponding ones found in sections \ref{section:Q<R}, \ref{section:separated} and \ref{section:diagonal}, and allows us to obtain
$$I_{\ell(Q)<\ell(R)} + I_{\textup{sep}} + I_{\textup{diag}} \lesssim  1.$$
Indeed, notice that in these sections things boil down to the martingale estimate
\begin{equation}\Big\| \Big(\sum_{Q \subset Q^*} |\Delta_Q f|^2 \Big)^{1/2} \Big\|_{L^q(\mu)} \lesssim \|f\|_{L^q(\mu)} = 1, \label{Appendix:MartinagleBound}
\end{equation}
which is easy for the classical martingales. These sections don't depend on the finer structure of the martingales.

The only difference lies in the treatment of the nested sum. Mostly it is much easier because of the simple martingales. But the thing that is more
complicated is that now only $f \in L^q(\mu)$ (and not bounded). The moral of the story: only the paraproduct requires a different argument.

\subsection{The paraproduct in $T1$}
We need to show that
\begin{displaymath}
\Big\| \Big( \mathop{\sum_{S \in \D}}_{S \subset Q^*} |\langle f \rangle_{S} |^2 A_S^2 \Big)^{1/2} \Big\|_{L^q(\mu)} \lesssim \mcar_V(q, 9),
\end{displaymath}
where
\begin{displaymath}
A_S(x)^2 :=\mathop{\sum_{R \in \mathcal{D}_{\textup{good}}}}_{R^{(r)}=S} 1_{R}(x) \int_{\ell(R)/2}^{\ell(R)}  \big| \theta_t 1(x)\big|^2\, \frac{dt}{t}.
\end{displaymath}
By Proposition \ref{CarlesonEmbeddingThm} it is enough to show the next lemma.
\begin{lem}
There holds that 
\begin{displaymath}
\Car_q ((A_S)_S) \lesssim \mcar_V(q, 9).
\end{displaymath}
\end{lem}
\begin{proof}
Let $Q \in \D$, $Q \subset Q^*$. We have that
\begin{align*}
\int_Q \Big[ \mathop{\sum_{S \in \D}}_{S \subset Q} A_S(x)^2 \Big]^{q/2}\,d\mu(x) &= \int_Q \Big[ \mathop{\sum_{S \in \D}}_{S \subset Q} \mathop{\sum_{R \in \mathcal{D}_{\textup{good}}}}_{R^{(r)}=S} 1_{R}(x) \int_{\ell(R)/2}^{\ell(R)}  \big| \theta_t 1(x)\big|^2\, \frac{dt}{t} \Big]^{q/2}\,d\mu(x) \\
&\le \int_Q \Big[ \mathop{\sum_{R \in \D}}_{d(R, Q^c) \ge 9\ell(R)} 1_{R}(x) \int_{\ell(R)/2}^{\ell(R)}  \big| \theta_t 1(x)\big|^2\, \frac{dt}{t} \Big]^{q/2}\,d\mu(x).
\end{align*}
Here we used that each appearing $R \in \D_{\textup{good}}$ satisfies that $R  \subset Q$ and $\ell(R) \le 2^{-r}\ell(Q)$. Therefore, we have that
$d(R, Q^c) \ge \ell(R)^{\gamma}\ell(Q)^{1-\gamma} \ge 2^{r(1-\gamma)} \ell(R) \ge 9\ell(R)$. Let $\mathcal{R}(Q)$ denote the maximal
$R \in \D$ for which $d(R, Q^c) \ge 9\ell(R)$.

We have reduced to bounding
\begin{align*}
 \int_Q & \Big[ \sum_{R \in \mathcal{R}(Q)} \mathop{\sum_{H \in D}}_{H \subset R} 1_{H}(x) \int_{\ell(H)/2}^{\ell(H)}  \big| \theta_t 1(x)\big|^2\, \frac{dt}{t} \Big]^{q/2}\,d\mu(x) \\
 &\le  \int_Q \Big[ \sum_{R \in \mathcal{R}(Q)} 1_R(x) \int_{0}^{\ell(R)}  \big| \theta_t 1(x)\big|^2\, \frac{dt}{t} \Big]^{q/2}\,d\mu(x) \\
 &= \sum_{R \in \mathcal{R}(Q)} \int_R  \Big[ \int_{0}^{\ell(R)}  \big| \theta_t 1(x)\big|^2\, \frac{dt}{t} \Big]^{q/2}\,d\mu(x) \\
 &\le \mcar_V(q, 9)^q \sum_{R \in \mathcal{R}(Q)} \mu(9R) \\
 & \lesssim \mcar_V(q, 9)^q \mu(Q).
\end{align*}
In these estimates we used the disjointness of the cubes in $\mathcal{R}(Q)$ and the bounded overlap property $\sum_{R \in \mathcal{R}(Q)} 1_{9R} \lesssim 1_Q$. We are done.
\end{proof}
This completes our proof of Theorem \ref{appendix:T1thm}.

\end{document}